\documentclass{amsart}

\usepackage{graphics}

\usepackage{amsfonts,amssymb,amscd,amsmath,latexsym,amsbsy,enumerate, tikz-cd,  ifthen, mathrsfs,mathtools,todonotes,color}

\usetikzlibrary{calc}

\def\C{\mathbb{C}}
\def\N{\mathbb{N}}

\def\cH{\mathcal H}

\newcommand{\SU}{\mathrm{SU}}
\newcommand{\U}{\mathrm{U}}

\def\cH{\mathcal H}

\newcommand{\D}{\mathcal{D}}
\newcommand{\MN}{M_N(\mathbb{C})}

\newtheorem{theorem}{Theorem}[section]
\newtheorem{definition}[theorem]{Definition}
\newtheorem{lemma}[theorem]{Lemma}
\newtheorem{corollary}[theorem]{Corollary}
\newtheorem{proposition}[theorem]{Proposition}
\newtheorem{example}[theorem]{Example}
\theoremstyle{definition}
\newtheorem{remark}[theorem]{Remark}

\numberwithin{equation}{section}

\begin{document}

\title[]{Ladder relations for a class of matrix valued orthogonal polynomials}

\author{Alfredo Dea\~{n}o}
\address{SMSAS, University of Kent (United Kingdom), Dpto. de Matem\'aticas, Universidad Carlos III de Madrid (Spain): 
 A.Deano-Cabrera@kent.ac.uk, alfredo.deanho@uc3m.es}
\author{ Bruno Eijsvoogel}
\address{IMAPP, Radboud Universiteit Nijmegen (The Netherlands), Department of Mathematics, KU Leuven (Belgium): b.eijsvoogel@math.ru.nl }
\author{Pablo Rom\'an}
\address{FaMAF-CIEM, Universidad Nacional de C\'ordoba (Argentina): roman@famaf.unc.edu.ar}

\begin{abstract}
	Using the theory introduced by Casper and Yakimov, we investigate the structure of algebras of differential and difference operators acting on matrix valued orthogo\-nal polynomials (MVOPs) on $\mathbb{R}$, and we derive algebraic and differential relations for these MVOPs. A particular case of importance is that of MVOPs with respect to a matrix weight  of the form $W(x)=e^{-v(x)}e^{xA}  e^{xA^\ast}$ on the real line, where $v$ is a scalar polyno\-mial of even degree with positive leading coefficient and $A$ is a constant matrix. 
	
\end{abstract}

\maketitle

\section{Introduction}
\label{sec:introduction}
Matrix valued orthogonal polynomials (MVOPs) were introduced by Krein in the 1940's and they appear in different areas of mathematics and mathematical physics, including spectral theory \cite{GroeneveltIK}, scattering theory \cite{Geronimo}, tiling problems \cite{DuitsK}, integrable systems \cite{AAGMM,AGMM,Manas1,IKR2} and stochastic processes \cite{IglesiaG,Iglesia,IglesiaR}.  There is also a fruitful interaction between harmonic analysis of matrix valued functions on compact symmetric pairs and matrix valued orthogonal polynomials. The first example of such an interaction is a family of matrix valued orthogonal polynomials related with the spherical functions of the compact symmetric pair $(\SU(3),\mathrm{S}(\U(2)\times \U(1))$, which appeared in \cite{GPT}.  Inspired by \cite{Koornwinder}, the case of $(\SU(2)\times \SU(2), \mathrm{diag})$ gave a direct approach \cite{KvPR1,KvPR2} leading to a general set-up in the context of multiplici\-ty free pairs \cite{HeckmanP,vPruijssenR}. In this context, some properties of the orthogonal polynomials such as orthogonality, recurrence relations and differential equations are understood in terms of the representation theory of the corresponding symmetric spaces, see also \cite{AKR1} for the quantum group case and \cite{KvPR3} for multivariable matrix orthogonal polynomials.

The interpretation of matrix valued orthogonal polynomials in terms of the representation theory of a certain symmetric pair is typically only for a limited (discrete) number of the parameters involved. It is then necessary to develop analytic tools to extend to a general set of parameters. In this context, shift operators for matrix valued orthogonal polynomials turned out to be very useful  \cite{CanteroMV2005,CanteroMV2007,KdlRR,IKR,IKR2}. 

In the last two decades, there has been  significant progress in understanding how the differential and algebraic properties of the classical scalar orthogonal polynomials can be extended to the matrix valued setting. A. Dur\'an and M. Ismail \cite{DI} introduced first order lowering and raising operators for MVOPs, and these results were rederived later on using the Riemann-Hilbert formu\-la\-tion by Gr\"unbaum  and coauthors \cite{GdIM}. This Riemann--Hilbert formulation is a powerful methodology to obtain algebraic and differential identities for MVOPs, as well as for the functions of the second kind, and it has been extensively used in the last few years, we refer the reader to \cite{BFM1,CM1,CM2}  and to  \cite{CMDdI}, \cite{CMDdI_Laguerre} for matrix orthogonal polynomials of Hermite and Laguerre type on the real line or the positive half-line. From the perspective of integrable systems, a very relevant result is the connection with matrix analogues of Painlev\'e equations for the recurrence coefficients, a theme that is well explored in the scalar case, see for instance the monograph \cite{vanAssche}.

There is also extensive work on orthogonal polynomial solutions of matrix valued differential equations of second order from an analytic point of view, we refer the reader for instance to \cite{Duran1,MR2595012,DuranG1,DurandlI2008,DG_2005b,DG_2005a}.

Very recently, Casper and Yakimov \cite{Casper2} developed a general framework to solve the matrix Bochner problem, that is, the classification of $N\times N$ weight matrices $W(x)$ whose associated MVOPs are eigenfunctions of a second order differential operator. The main purpose of this paper is to apply the theory proposed in \cite{Casper2}  to MVOPs defined on the real line. This approach is an alternative to the Riemann--Hilbert methodology and has the advantage of being more transparent in the derivation of differential and difference identities for MVOPs.

Given $N\in \N$, we denote by $M_N(\C)$ the space of all $N\times N$ matrices with complex entries. Let $W \colon \mathbb{R}\to M_N(\C)$ be a positive definite matrix weight supported in the (possibly infinite) interval $[a,b]$. For $M_N(\C)$-valued functions $H,G$, we define the matrix valued inner product
\begin{equation}
\label{eq:left_inner_prod}
\langle H, G \rangle = \int_a^b H(y) W(y) G(y)^\ast \, dy \in M_N(\C).
\end{equation}
Using standard arguments it can be shown that there exists a unique sequence $(P(x,n))_n$ of monic matrix valued orthogonal polynomials (MVOPs) with respect to $W$, in the following sense:
\begin{equation}\label{Hn}
\langle P(x,n),P(x,m) \rangle = \mathcal{H}(n) \delta_{n,m},
\end{equation}
where the squared norm $\mathcal{H}(n)$ is a positive definite matrix, see for instance \cite{Damanik,GT}. As a direct consequence of orthogonality, the polynomials $P(x,n)$ satisfy the following three-term recurrence relation
\begin{equation}
\label{eq:three_term_monic}
xP(x,n)= P(x,n+1) + B(n) P(x,n) + C(n) P(x,n-1),
\end{equation}
where $B(n),C(n)\in M_N(\mathbb{C})$. Note that these matrix coefficients multiply the MVOPs from the left. From the orthogonality relations, we also obtain that
\[
B(n)=X(n)-X(n+1),\qquad C(n)=\mathcal{H}(n)\mathcal{H}(n-1)^{-1},
\]
where $X(n)$ is the one-but-leading coefficient of $P(x,n)$, i.e. $P(x,n)=x^n+x^{n-1}X(n)+\cdots$.

We note that the previous MVOPs can be related to the matrix biorthogonal polynomials presented recently in \cite{BFM1}, namely the Hermitian case (Section 2.4) since the weight satisfies $W(x)=W(x)^*$. As a consequence, the MVOPs that we study coincide with $P_n^{\rm L}$ in their notation.

The structure of this paper is the following:  in Section \ref{sec:pre}, following the approach of Casper and Yakimov in \cite{Casper2}, we discuss differential and difference operators for these MVOPs. In this noncommutative setting operators can act both from the right and from the left. We consider two isomorphic algebras of operators acting on MVOPs, one algebra of matrix valued differential operators acting from the right, $\mathcal{F}_R(P)$, and a second algebra of matrix valued discrete operators acting from the left, $\mathcal{F}_L(P)$. In this construction, a differential operator $\D\in\mathcal{F}_R(P)$ acts naturally on the variable of the MVOPs, whereas a difference operator $M\in\mathcal{F}_L(P)$ acts on its degree.

In Section \ref{sec:general-weights} we fix the form of two differential operators $\mathcal{D}=\partial_x+A$, with $A$ an arbitrary matrix, and $\mathcal{D}^\dagger=-\mathcal{D}+v'(x)$, with $v(x)$ a scalar polynomial, and we work out the corresponding difference operators $M$ and $M^\dagger$ using the techniques in Section \ref{sec:pre}; the form of these operators is motivated by the exponential type weights on $\mathbb{R}$ that appear later on the paper, but the results that we obtain in Section \ref{sec:general-weights}  hold in a more general setting, just by prescribing the form of the operators. We also investigate the structure of the Lie algebra generated by $\mathcal{D}$ and $\mathcal{D}^\dagger$, as well as discrete string equations for the recurrence coefficients of the corresponding family of MVOPs.

The approach proposed in \cite{Casper2} is particularly explicit in the case of exponential weights defined on the real line; these weights are studied in Section 
\ref{sec:exponential-weights} and written in the form 
$W(x)=e^{-v(x)}e^{xA} e^{xA^\ast}$, with $x\in\mathbb{R}$, where the potential $v(x)$ is an even polynomial with positive leading coefficient and $A$ is a constant matrix. In this case, the differential operator $\D=\partial_x+A$ has an uncomplicated adjoint $\D^\dagger=-\D+v'(x)$, with respect to the matrix valued inner product given by $W$. By adjoint we mean that
\begin{equation}\label{eq:adjointdef}
\langle P\cdot \D,Q\rangle = \langle P,Q\cdot \D^\dagger\rangle,
\end{equation}
for all matrix-valued polynomials $P$ and $Q$. The actions of $\D$ and $\D^\dagger$ on the 
MVOPs are
\begin{align*}
(P \cdot \D)(x,n)&=P'(x,n)+P(x,n)A,\\ 
(P \cdot \D^\dagger)(x,n)&=-P'(x,n)-P(x,n)A - v'(x)P(x,n),
\end{align*} 
which will imply that $\D,\D^\dagger\in\mathcal{F}_R(P)$. Our first result states that $\D, \D^\dagger$ induce ladder relations:
\[
P\cdot \D(x,n) = \sum_{j=-k+1}^{0} A_j(n) P(x,n+j),\quad P\cdot \D^\dagger(x,n) = \sum_{j=0}^{k-1} \widetilde{A}_j(n) P(x,n+j),
\]
where $k = \deg v$, with some matrix coefficients $A_j(n)$ and $\widetilde{A}_j(n)$. These operators are closely related to the creation and annihilation operators given in \cite{DI}, with the advantage that $\D$ and $\D^\dagger$ are each other's adjoint. This property is crucial to show that the Lie algebra generated by the operators $\D$ and $\D^\dagger$ is finite dimensional and it is isomorphic to the algebra generated by the ladder operators for the scalar weight $w(x)=e^{-v(x)}$, see for instance \cite{ChenIsmail}, \cite[Chapter 3]{Ismail}. From the ladder relations, we obtain nonlinear algebraic equations for the coefficients of the recurrence relation \eqref{eq:three_term_monic}. In the literature these identities are often called discrete (or Freud) string equations. We include two examples: Hermite-type weights with $v(x)=x^2+tx$ and $t\in\mathbb{R}$, and Freud-type weights with $v(x)=x^4+tx^2$, and in this last case the discrete string equations can be seen as a matrix analogue of the discrete Painlev\'e I equation \cite{vanAssche}. We remark that this kind of identity, which is very relevant in integrable systems, is obtained here as a result of the relation between the two Fourier algebras of operators and in particular from the fact that $\mathcal{F}_L(P)$ and $\mathcal{F}_R(P)$ are isomorphic.

Section \ref{sec:hermite} is devoted to the detailed study of Hermite-type matrix valued weights. In this setting, we show first that the ladder relations, written in terms of the squared norms of the monic MVOPs, in fact characterize this matrix valued weight. Next, for a Hermite-type weight of the form $W(x)=e^{-x^2} L(x)L(x)^\ast$, with $L$ a lower triangular matrix constructed from scalar Hermite polynomials, we find a second order differential operator $D$ that, together with $\mathcal{D}$, $\mathcal{D}^\dagger$ and $I$, generate a Lie algebra of dimension 4 known as the Harmonic oscillator algebra. Using the Casimir $\mathcal{C}$ of this Lie algebra, we can diagonalize the norms $\mathcal{H}(n)$ of the MVOPs and combining this with a ladder relation characterisation we propose a computational method for these Hermite MVOPs that is more efficient that the standard Gram--Schmidt procedure.

In Section \ref{sec:Pearson} we further specify a  Hermite type weight in such a way that there exists a matrix valued Pearson equation for the weight $W$, for specific choices of the matrix $A$. This setting gives extra ladder relations for the corresponding MVOPs.

Complementing the previous results, in Section \ref{sec:Toda} we investigate similar identities of differential and algebraic type for a deformation of  the matrix weight with respect to extra parameters. Examples include the non-Abelian Toda and Langmuir lattice equations which appear for instance in \cite{BFGA, Bruschi, Gekhtman1, Gekhtman2}.
For the particular case of a multi-time Toda deformation, we give a Lax pair formulation, analogous to \cite[(2.8.5)]{Ismail} for the scalar case.

In the appendix we establish the link between the ladder relations obtained with this metho\-do\-logy and the ladder operators previously considered by A. Dur\'an and M. Ismail in \cite{DI}.\\

\noindent
\textbf{Acknowledgements.} The authors would like to thank Erik Koelink and Mourad E. H. Ismail for fruitful discussions about the content and scope of this paper. Bruno Eijsvoogel thanks Riley Casper, Koen Reijnders, John van de Wetering and Walter Van Assche for useful discussions as well. The authors would like to thank the three anonymous referees for their useful remarks and corrections, that led to an improved version of the manuscript.

The support of Erasmus+ travel grant and EPSRC grant ``Painlev\'e equations: analytical properties and numerical computation", reference EP/P026532/1 is grate\-fully acknowledged. Alfredo Dea\~{n}o acknowledges additional financial support from the London Mathematical Society (Research in pairs scheme) for a visit to RU Nij\-megen in June--July 2019. The work of Pablo Rom\'an was supported by a FONCyT grant PICT 2014-3452 and by SeCyTUNC.

\section{Preliminaries}
\label{sec:pre}
In this section we introduce the left and right Fourier algebras related to the sequence of monic MVOPs, following a recent work of Casper and Yakimov \cite{Casper2}. Some of the results in this section have already appeared in a more general form in \cite{Casper2}, but we include them to keep our description self-contained. 

We view the sequence $P(x,n)$ as a function $P:\mathbb{C}\times \mathbb{N}_0 \to M_N(\mathbb{C})$. It is, therefore, natural to consider the space of functions
$$\mathcal{P}=\{ Q:\mathbb{C}\times \mathbb{N}_0 \to M_N(\mathbb{C}) \colon \quad Q(x,n) \textrm{ is rational in $x$ for fixed $n$} \}.$$
A differential operator of the form
\begin{equation}
\label{eq:DifferentialOperator}
\D=\sum_{j=0}^n \partial_x^j F_j(x), \qquad \partial_x^j := \frac{d^j}{dx^j},
\end{equation}
where $F_j:\mathbb{C}\to M_N(\mathbb{C})$ is a rational function of $x$, acts on an element $Q\in\mathcal{P}$ from the right by
$$(Q\cdot \D)(x,n)  = \sum_{j=0}^n (\partial_x^jQ)(x,n)\,  F_j(x).
$$
We denote the algebra of all differential operators of the form \eqref{eq:DifferentialOperator} by $\mathcal{M}_N$. Now we consider a left action on $\mathcal{P}$ by discrete operators. For $j\in\mathbb{Z}$, let $\delta^{j}$ be the discrete operator which acts on a sequence $A:\mathbb{N}_0 \to M_N(\mathbb{C})$ by
$$(\delta^j \cdot A)(n)=A(n+j),$$
where we take the value of a sequence at a negative integer to be equal to the zero matrix. A discrete operator
\begin{equation}
 \label{eq:DifferenceOperator}
M=\sum_{j=-\ell}^k A_j(n) \delta^j,
 \end{equation}
where $A_{-\ell},\ldots,A_k$ are sequences, acts on elements of $\mathcal{P}$ from the left by
\begin{align*}
(M \cdot Q)(x,n) &= \sum_{j=-\ell}^k A_j(n) \, (\delta^j\cdot Q)(x,n) = \sum_{j=-\ell}^k A_j(n) \, Q(x,n+j).
\end{align*}

We denote the algebra of all discrete operators of the form \eqref{eq:DifferenceOperator} by $\mathcal{N}_N$, and we adapt the cons\-truc\-tion given in 
\cite[Definition 2.20]{Casper2} to our setting:
\begin{definition}
For the sequence $(P(x,n))_{n}$ of MVOPs we define:
\begin{equation}
\label{eq:definition-Fourier-algebras}
\begin{split}
\mathcal{F}_L(P)&=\{ M\in \mathcal{N}_N \colon \exists \D\in \mathcal{M}_N,\, M\cdot P = P\cdot \D \} \subset \mathcal{N}_{N},\\
 \mathcal{F}_R(P)&=\{ \D\in \mathcal{M}_N \colon \exists M\in \mathcal{N}_N,\, M\cdot P = P\cdot \D \}\subset \mathcal{M}_{N}.
\end{split}
\end{equation}
\end{definition}

Using these Fourier algebras, we prove the following uniqueness result:
\begin{lemma}
\label{lem:isomorphism}
Given $\D\in\mathcal{F}_R(P)$, there exists a unique $M\in\mathcal{F}_L(P)$ such that
 $M\cdot P = P\cdot \D$. Conversely, given $M\in\mathcal{F}_L(P)$, there exists a unique $\D\in\mathcal{F}_R(P)$ such that $M\cdot P = P\cdot \D$.
\end{lemma}
\begin{proof}
 Let us assume that there exist $M_1, M_2 \in \mathcal{F}_L(P)$ such that
\[
(M_1\cdot P)(x,n) = (P\cdot \D)(x,n), \qquad (M_2 \cdot P)(x,n) = (P\cdot \D)(x,n),
\]
then  $((M_1 - M_2)\cdot P)(x,n)=0$. Suppose that $M_1-M_2$ has the following expression
\begin{equation}
 \label{eq:M1-M2}
((M_1-M_2)\cdot P)(x,n) = \sum_{j=-\ell}^k A_j(n) \, P(x,n+j).
 \end{equation}
By taking the leading coefficient of \eqref{eq:M1-M2} we obtain that $A_k(n)=0$. Proceeding recursively we conclude that $A_j(n)=0$ for all $j=-\ell,\ldots, k$. The converse is proven in a similar way.
\end{proof}

It follows directly from the definition that the elements of $\mathcal{F}_L(P)$ are related 
to the elements of $\mathcal{F}_R(P)$. Lemma \ref{lem:isomorphism} shows that the map
\begin{equation*}
\varphi\colon \mathcal{F}_L(P) \to \mathcal{F}_R(P),\qquad \text{ defined by }\quad M\cdot P = P \cdot \varphi(M),
\end{equation*}
is in fact a bijection. In \cite{Casper2} this map is called the \textit{generalized Fourier map}. As in \cite{Casper2}, we introduce the bispectral algebras $\mathcal{B}_L(P)$ and $\mathcal{B}_R(P)$:
\begin{equation}
\label{eq:definition-bispectral}
 \begin{split}
 \mathcal{B}_L(P)&=\{ M\in \mathcal{F}_L(P)\colon \, \mathrm{order}(\varphi(M)) =0 \},\\
 \mathcal{B}_R(P)&=\{ D\in \mathcal{F}_R(P)\colon \, \mathrm{order}(\varphi^{-1}(D)) =0 \},
 \end{split}
\end{equation}
where a differential operator of order zero is a rational function $F:\mathbb{C}\to M_N(\mathbb{C})$ and a discrete operator of order zero is a sequence $A:\mathbb{N}_0\to M_N(\mathbb{C})$.
\begin{remark}
For $M_{1}, M_{2}\in \mathcal{F}_{L}(P)$ we have that
\begin{equation}
\label{eq:FLalgebra}
M_{1}M_{2}\cdot P = M_{1}\cdot P \cdot \varphi(M_{2}) = P \cdot \varphi(M_{1})\varphi(M_{2}),
\end{equation}
which implies that $M_{1}M_{2}\in \mathcal{F}_{L}(P)$. Therefore the linear space $\mathcal{F}_{L}(P)$ is a subalgebra of $\mathcal{N}_{N}$. A similar computation shows that $\mathcal{F}_{R}(P)$ is an algebra. We shall refer to $\mathcal{F}_{L}(P)$  and $\mathcal{F}_{R}(P)$ as the left and right Fourier algebras respectively.

Now it follows from \eqref{eq:FLalgebra} that $M_{1}M_{2}\cdot P = P \cdot \varphi(M_{1})\varphi(M_{2})$ for all $M_{1},M_{2}\in \mathcal{F}_{L}(P)$. On the other hand, by the definition of $\varphi$, we have that $M_{1}M_{2}\cdot P = P\cdot \varphi(M_{1}M_{2})$ and, since $\varphi$ is bijective, we conclude that $\varphi$ is an isomorphism of algebras.

In \cite{Casper2} it is shown that this map $\varphi$ is an isomorphism of algebras
in a more general setting as well. The crucial requirement there is that
$P$ only has trivial left and right annihilators.
\end{remark}

\begin{remark}
\label{rmk:three-term-Fourier}
We can write the three-term recurrence relation \eqref{eq:three_term_monic} as
$$xP = P\cdot x = L\cdot P, \qquad \text{where } \quad L=\delta + B(n) + C(n)\delta^{-1}.$$
Therefore $x\in \mathcal{F}_R$, $L\in \mathcal{F}_L$ and $\varphi(L)=x$. Moreover, for every polynomial $v \in \mathbb{C}[x]$, we have
\begin{equation*}
P\cdot v(x) =P\cdot v(\varphi(L))= v(L)\cdot P.
\end{equation*}
\end{remark}

The main result from \cite{Casper2} that we use in this paper is the existence of an adjoint operation $\dagger$ in the Fourier algebras $\mathcal{F}_L(P)$ and $\mathcal{F}_R(P)$, see \cite[\S 3.1]{Casper2}.
In order to introduce the adjoint on $\mathcal{F}_L(P)$, we first note that the algebra of discrete operators $\mathcal{N}_N$ has a $\ast$-operation given by
\begin{equation}\label{eq:star}
\left( \sum_{j=-\ell}^k A_j(n) \, \delta^j \right)^\ast = \sum_{j=-\ell}^k A_j(n-j)^\ast \, \delta^{-j},
\end{equation}
where $A_j(n-j)^\ast$ denotes the conjugate transpose of $A_j(n-j)$. The adjoint of $M\in \mathcal{N}_N$ is 
\begin{equation}
 \label{eq:adjointM}
 M^\dagger = \mathcal{H}(n) M^\ast \mathcal{H}(n)^{-1},
\end{equation}
where the squared norm $\mathcal{H}(n)$, given by \eqref{Hn}, is viewed as a sequence. The following relation holds:
$$\langle (M\cdot P)(x,n),P(x,m)\rangle = \langle P(x,n),(M^\dagger \cdot P)(x,m)\rangle.$$

In \cite{GT}, A. Gr\"unbaum and J. Tirao introduce an adjoint in the algebra of all differential operators having the orthogonal polynomials as eigenfunctions. This was recently extended in \cite[Corollary 3.8]{Casper2} where the authors show that for every differential operator $D\in  \mathcal{F}_R(P)$ there exists a unique operator $\D^\dagger\in  \mathcal{F}_R(P)$ such that
$$\langle P\cdot \D, Q \rangle = \langle P,Q\cdot \D^\dagger \rangle,$$
for all $P,Q\in M_N(\mathbb{C})[x]$. We say that $\D^\dagger$ is the adjoint of $\D$. 
Moreover, $\mathcal{F}_L(P)$ is closed under the adjoint operation $\dagger$ and $\varphi(M^\dagger) = \varphi(M)^\dagger$ for all $M\in \mathcal{F}_L(P)$.

\begin{definition}
Given a pair $(M,\D)$ with $M\in \mathcal{F}_L(P)$ and $\D\in \mathcal{F}_R(P)$, a relation of the form
\begin{equation*}
M\cdot P = P \cdot \D,
\qquad
\text{ where } \quad M=\sum_{j=-\ell}^k A_j(n) \, \delta^j.
\end{equation*}
is called a ladder relation. If the operator $M$ only contains nonpositive  (non\-negative) powers of $\delta$, we say that it is a lowering (raising) relation.
\end{definition}
Observe that if a pair $(M,\D)$ gives a raising relation, then it follows from \eqref{eq:adjointM} and \eqref{eq:star} that $(M^\dagger,\D^\dagger)$ gives a lowering relation and viceversa.

\section{Ladder relations for MVOPs with prescribed differential operators}
\label{sec:general-weights}
In this section we study ladder relations for MVOPs in a general setting, where we only prescribe the form of the differential operators $\mathcal{D}$ and $\mathcal{D}^{\dagger}$. Our choice is motivated by the exponential weights on the real line that we study next in Section \ref{sec:exponential-weights}, but we emphasise that the results that we obtain hold in a more general setting. First we need the following notation: 
\begin{remark} 
Given the difference operator $L$ corresponding to the three-term recurrence relation, see 
\eqref{rmk:three-term-Fourier}, and any polynomial $q$, we denote by $(q(L))_{j}(n)$ the coefficient of the difference operator $q(L)$ of order $j$ in $\delta$. In other words, we have
\begin{equation}
\label{rmk:notation}
q(L)=\sum_{j=-\deg q}^{\deg q} (q(L))_{j}(n) \, \delta^j.
\end{equation}
The calculation of $(q(L))_{j}(n)$ can be carried out following the scheme shown in Figure \ref{fig:pLj}: $(q(L))_{j}(n)$ is equal to the sum over all possible paths from $P(x,n)$ to $P(x,n+j)$ in $\deg q$ steps, where in each path we multiply the coefficients corresponding to each arrow.

\begin{figure}[!h]
\begin{center}
\begin{tikzcd}[row sep=normal, column sep=normal]
\cdots	\arrow[shift left]{r}{C(n+2)}
& P(x,n+1)  	\arrow[shift left]{r}{C(n+1)}
		\arrow[shift left]{l}{I}
	         \arrow[out=120, in=60, loop]{r}{B(n+1)}
& P(x,n)
		\arrow[shift left]{l}{I}
		\arrow[shift left]{r}{C(n)}
	         \arrow[out=120, in=60, loop]{r}{B(n)}
& P(x,n-1) \arrow[shift left]{l}{I}
		\arrow[shift left]{r}{C(n-1)}
	         \arrow[out=120, in=60, loop]{r}{B(n-1)}
& \cdots 	 \arrow[shift left]{l}{I}
\end{tikzcd}
\caption{Scheme for the calculation of $(q(L))_{j}(n)$.}
\label{fig:pLj}
\end{center}
\end{figure}
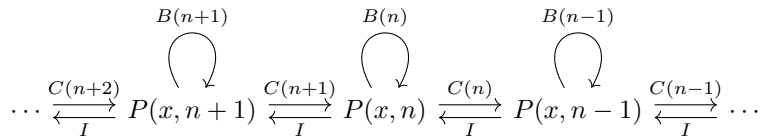

For example, if $q(x)=x^3$ we have $q(L)=L^3$, and in order to compute $\left(L^3 \right)_{-1}(n)$ we have a total of six paths from $P(x,n)$ to $P(x,n-1)$ in three steps:
\begin{multline*}
\left(L^3 \right)_{-1}(n)
=
C(n)B(n-1)^2+B(n)C(n)B(n-1)+B(n)^2C(n)\\
+C(n+1)C(n)+C(n)C(n-1)+C(n)^2.
\end{multline*}

\end{remark}

\begin{theorem}
\label{lem:low}
Let $W$ be a matrix weight with monic MVOPs $P(x,n)$ such that
\begin{equation}\label{DDdagger}
\D = \partial_x + A,\qquad  \text{and} \qquad \D^\dagger = -\D + v'(x),
\end{equation}
for some polynomial $v(x)$ of degree $k$ and some constant matrix $A$.
Then the monic polynomials $P(x,n)$ satisfy the lowering relation
\begin{equation}\label{DM}
P\cdot \D = M\cdot P,\qquad M=\sum_{j=-k+1}^{0} A_j(n) \delta^{j},
\end{equation}
where, using the notation \eqref{rmk:notation}, we have
\begin{equation}\label{Ajv}
A_0(n)=A,
\qquad
A_j(n) =
\left(v'(L)\right)_{j}(n),  \quad j\neq 0.
\end{equation}
\end{theorem}

\begin{proof}
It follows from \cite[Theorem 3.7]{Casper2} that $\D\in \mathcal{F}_R(P)$. Notice that $(P\cdot \partial_x)(x,n)$ is a polynomial of degree $n-1$. Furthermore $(P\cdot\D)(x,n)$ is
a polynomial of degree $n$ with $A$ as its leading coefficient. Therefore
\begin{equation*}
(P\cdot\D)(x,n)
=
\sum_{j=-n}^0 A_j(n)\left( \delta^j \cdot P\right)(x,n). 
\end{equation*}
Using the definition of $\mathcal{D}$ given in \eqref{eq:def-operatorD}, it is clear that $A_0(n)=A$. Moreover for $j<0$, we have 
\begin{equation*}
\begin{aligned}
A_j(n) &= \langle P\cdot\D, \delta^j \cdot P \rangle \cH(n-j)^{-1}=
\langle P, \delta^j \cdot P \cdot \D^\dagger \rangle \cH(n-j)^{-1}\\
&=\langle P, \delta^j \cdot P \cdot v'(x) \rangle \cH(n-j)^{-1}=
\langle P \cdot v'(x), \delta^j \cdot P  \rangle \cH(n-j)^{-1}
\\&=\langle  v'(L)\cdot P , \delta^j \cdot P  \rangle \cH(n-j)^{-1},
\end{aligned}
\end{equation*}
where we have used that $\langle P, \delta^j \cdot P \cdot \D \rangle = 0$ for $j<0$
in the third equality, and the fact that $v'(x)$ is a scalar function in the fourth equality.
Using \eqref{rmk:notation} we get
\begin{equation*}
A_j(n) = \left(v'(L)\right)_{j}(n).
\end{equation*}
In order to complete the proof we note that $\left(v'(L)\right)_{j}(n)=0$ for all $j\leq -k$.
\end{proof}

\begin{corollary}\label{cor:adjoint}
There exists a unique $M^\dagger\in \mathcal{F}_L(P)$ such that
\begin{equation}
\label{eq:nocom}
M\cdot P
=
P\cdot \D
\qquad
\textrm{and}
\qquad
M^\dagger\cdot P
=
P\cdot \D^\dagger.
\end{equation}
Moreover, we can write 
\begin{equation}
M^{\dagger}
=
\mathcal{H}(n)\sum_{i=0}^{k-1}A_{-i}(n+i)^* \mathcal{H}(n+i)^{-1}\delta^i,
\end{equation}
in terms of the coefficients of $M$ in \eqref{DM} and the squared norms of the MVOPs given by \eqref{Hn}.
\end{corollary}
\begin{proof}
The result follows directly from formulas \eqref{eq:star} and \eqref{eq:adjointM} applied to the operator $M$.
\end{proof}
If a Riemann-Hilbert formulation for matrix orthogonal polynomials can be used, then the previous theorem can be compared with the results in  \cite{GdIM}. We note that our approach gives an elementary proof for the lowering relation and we also obtain the exact degree of the lowering operator $M$ and the raising operator $M^{\dagger}$.

Next we investigate the properties of the Lie algebra generated by the operators $\mathcal{D}$ and $\mathcal{D}^{\dagger}$ constructed before. Using the explicit expressions of $\D$ and $\D^\dagger$, we find that
\begin{equation}
\label{eq:conmutatorsDDdagger}
\left[\D^\dagger,\D\right]=v''(x), \qquad \D+\D^\dagger=v'(x).
\end{equation}

In the scalar case $A=0$, $N=1$, the differential operators $\D $ and $\D^\dagger$ generate a finite dimensional Lie algebra, see for instance \cite[Theorem 3.1]{ChenIsmail}. Moreover it is conjectured in \cite[\S 24.5]{Ismail} that this is a characterizing property of the weight. In the following proposition, we prove that in the matrix valued setting, the operators $\D$ and $\D^\dagger$ generate a finite dimensional Lie algebra  $\mathfrak{g}$ which is  independent of the matrix $A$ and is isomorphic to the Lie algebra corresponding  to the scalar case.

\begin{proposition}
\label{thm:LieAlgebra}
 The differential operators $\D$ and $\D^\dagger$ generate a Lie algebra $\mathfrak{g}$ of dimension $k+1$.
\end{proposition}
\begin{proof}
Let $v^{(j)}$ be the $j$-th derivative of $v$. Using that $v^{(j)}$ is scalar, and so it commutes with the matrix $A$, we first observe that $[\D,v^{(j)}]=-v^{(j+1)}$. Then for any $M_N(\C)$-valued smooth function $F$ we have
$$F\cdot [\D,v^{(j)}] = F\cdot \D v^{(j)} - (Fv^{(j)})\cdot \D=-v^{(j+1)}F.$$
Since $\D^\dagger= -\D +v'(x)$, we obtain that the Lie algebra generated by $\D$ and $\D^\dagger$ is generated by $\{\D,v'(x),\ldots,v^{(k)}\},$ and is, therefore, $(k+1)$-dimensional.
\end{proof}

\begin{remark}
If we take $v(x)=x^2$, the Lie algebra $\mathfrak{g}$ generated by $\D$ and $\D^\dagger$ is $3$-dimensional and we have the following relations
\begin{equation}
\label{eq:wehave}
\D+\D^\dagger = 2x+t,
\qquad
\left[\D^\dagger,\D\right]=2.
\end{equation}
In this case, $\mathfrak{g}$ is isomorphic to the Lie algebra of the 3-dimensional Heisenberg group, which can be identified with the $3\times 3$ strictly upper triangular matrices. We can map the operators $\mathcal{D},\mathcal{D}^\dagger$ and the identity to a basis of the Lie algebra as follows:
\begin{equation*}
\begin{aligned}
\mathcal{D}
& \longleftrightarrow \begin{pmatrix} 0 & 1 & 0\\0 & 0 & 0\\ 0 & 0 &0\end{pmatrix}, \qquad
\mathcal{D}^{\dagger}
 \longleftrightarrow \begin{pmatrix} 0 & 0 & 0\\0 & 0 & 1\\ 0 & 0 &0\end{pmatrix}, \qquad
I
 \longleftrightarrow \begin{pmatrix} 0 & 0 & 1/2\\0 & 0 & 0\\ 0 & 0 &0\end{pmatrix}.
\end{aligned}
\end{equation*}
In Section \ref{sec:harmonic} we extend $\mathfrak{g}$ to a 4-dimensional Lie algebra $\mathfrak{h}$ with a nontrivial Casimir element in the center of $\mathcal{U}(\mathfrak{h})$. This operator induces a new difference operator having the MVOPs as eigenfunctions. Whether $\mathfrak{g}$ can be exploited in a similar way in the case of $v(x)$ of degree greater than two will require further investigation.
\end{remark}

From the previous results, we obtain nonlinear relations for the coefficients of the three-term recurrence relation. These identities can be seen as a non-Abelian analogue of the discrete string or Freud equations, see for instance \cite[\S 4.1.1.5]{Bleher}. In order to obtain them, we observe that the second equation of \eqref{eq:conmutatorsDDdagger} and Remark \ref{rmk:three-term-Fourier} imply that 
\begin{equation*}
\varphi^{-1}(\D+\D^\dagger) = M+M^\dagger = v'(L) .
\end{equation*}
explicitly in terms of the difference operator coming from the three-term recurrence relation. Using the explicit formula for $\D$ in Proposition \ref{lem:low} and the definition of $M^\dagger$ in \eqref{eq:adjointM}, we verify
\begin{equation}
\label{eq:zerocoeff}
\begin{aligned}
(v'(L))_{0} (n)
=
\left(\varphi^{-1}(\D+\D^\dagger)\right)_{0} (n)
&=
\left(M+M^\dagger\right)_{0}(n)\\
&=
A+\mathcal{H}(n)A^* \mathcal{H}(n)^{-1},
\end{aligned}
\end{equation}
using the notation \eqref{rmk:notation} again. 
\begin{theorem}
\label{lem:commute}
Let $W$ be a matrix weight with monic MVOPs $P(x,n)$ such that
\begin{equation*}
\D = \partial_x + A \in \mathcal{F}_R(P),\qquad  \text{and} \qquad \D^\dagger = -\D + v'(x),
\end{equation*}
for some polynomial $v(x)$ of degree $k$. Then the coefficients of the three term recurrence relation for $P(x,n)$ satisfy the following discrete string relations:
\begin{equation}\label{eq:ladder-in-theorem}
\begin{aligned}
\left[B(n),A\right]
&=
I + \left(v'(L)\right)_{-1}(n)
-
\left(v'(L)\right)_{-1}(n+1),\qquad n \geq 0,
\\
\left[C(n),A\right]
&=
C(n)\left(v'(L)\right)_{0}(n-1)
-
\left(v'(L)\right)_{0}(n)C(n),\qquad n \geq 1.
\end{aligned}
\end{equation}
\end{theorem}

\begin{proof}
The coefficients of $x^{n-1}$ in 
\begin{equation*}
\left(P\cdot \mathcal{D}\right)(x,n) =\left( M\cdot P\right)(x,n).
\end{equation*}
give
\begin{equation} \label{eq:earlier}
nI + X(n)A
=
A X(n) + A_{-1}(n).
\end{equation}
Taking the difference of \eqref{eq:earlier} for $n+1$ and $n$ we obtain 
\begin{equation*}
[B(n),A]  - I
=
A_{-1}(n) - A_{-1}(n+1),
\end{equation*}
which together with Proposition \ref{lem:low} gives the first desired result.

For the second commutation relation we take \eqref{eq:zerocoeff} with $n$ replaced by $n-1$ and 
multiplied by $C(n)$ from the left, and we subtract from it \eqref{eq:zerocoeff} with parameter $n$ and multiplied by $C(n)$ from the right. The result follows after cancellation of $\cH(n)A^*\cH(n-1)^{-1}$ terms.
\end{proof}

\begin{example}\label{ex:Hermiteweight} 
Let $v(x)=x^2+tx$ and $A$ a generic matrix, then \eqref{DDdagger} gives the operators
\begin{equation}
\label{eq:HermiteDD}
\D=\partial_x+A,\qquad
\D^\dagger = -\partial_x -A +2 x+t.
\end{equation}

Using Theorem \ref{lem:low} with $v'(x)=2x+t$, we obtain 
\[
A_{-1}(n)=2C(n)=2\cH(n)\cH(n-1)^{-1}.
\] 
Therefore
\begin{equation}
\label{eq:M-Mdagger}
M= A + 2C(n)\delta^{-1} , \qquad M^\dagger=
2 \delta
+\mathcal{H}(n)A^\ast \mathcal{H}(n)^{-1}  = 2 \delta
+2B(n)-A + t I,
\end{equation}
using \eqref{eq:zerocoeff} and the fact that $(v'(L))_{0} (n)=2B(n)+tI$ in the last equality.  

The discrete string equations from Theorem \ref{lem:commute} are
\begin{equation}\label{eq:res1}
\begin{aligned}
\left[B(n),A\right]
&=
2\left(C(n) - C(n+1)\right) + I,
\qquad C(0) \vcentcolon = 0\\
\left[C(n),A\right]
&=
2\left(C(n)B(n-1)-B(n)C(n)\right).
\end{aligned}
\end{equation}

Additionally, if we sum the first identity from $0$ to $n-1$, we obtain
\begin{equation*}
\sum_{k=0}^{n-1} \left[B(k),A\right] = n - 2C(n).
\end{equation*}
since $C(0)=0$.
\end{example}

\begin{example}\label{ex:Freudweight} 
Let $v(x)=x^4+tx^2$ and  $A$ a generic matrix, then \eqref{DDdagger} gives the operators
\begin{equation*}
\D=\partial_x+A,\qquad \D^\dagger = -\partial_x -A +4 x^3+2tx.
\end{equation*}
The following relations hold true:
\begin{equation*}
\left[\D^\dagger,\D\right]=12x^2 + 2t,
\quad
\left[\left[\D^\dagger,\D\right],\D\right]=24x,
\quad
\left[\left[\left[\D^\dagger,\D\right],\D\right],\D\right]=24.
\end{equation*}

The relation $P\cdot \D=M\cdot P$ in Theorem \ref{lem:low} is written explicitly as follows:
\begin{equation*}
P'(x,n)+P(x,n)A=\left(A+A_{-1}(n)\delta^{-1}
+ A_{-2}(n)\delta^{-2}
+ A_{-3}(n)\delta^{-3}\right)P(x,n).
\end{equation*}
where the coefficients are computed using that $A_{j}(n)=(v'(L))_{j}(n)$ and the scheme in Figure \ref{fig:pLj}:
\begin{equation*}
\begin{aligned}
A_{-1}(n)
&=
4 \Bigl(C(n) C( n-1)+ C(n)^2+ C(n+1) C(n)+ B(n)^2 C(n)
\\
&\qquad+ B(n)C(n)C( n-1)+ C(n)B(n-1)^2\Bigr)
+ 2 t C(n)
\\
A_{-2}(n)
&=
 4 \left(B (n)C(n)C(n-1)+C(n )B(n-1)C(n-1)\right.\\
 &\left.+C(n)C(n-1)B(n-2)\right),
\\
A_{-3}(n)
&=
4 C(n)C(n-1)C(n-2).
\end{aligned}
\end{equation*}
Furthermore, we use Theorem \ref{lem:commute} to compute the discrete string equations
\begin{equation*}
\begin{aligned}
\left[C(n),A\right]
&=
 C(n)\left(v'(L)\right)_{0}(n-1)
 - \left(v'(L)\right)_{0}(n)C(n),
 \\
[B(n),A]
&=
I+ \left(v'(L)\right)_{-1}(n) - \left(v'(L)\right)_{-1}(n+1).
\end{aligned}
\end{equation*}
If we replace $\left(v'(L)\right)_{0}(n)=A+\mathcal{H}(n)A^* \mathcal{H}(n)^{-1}$ in the first equation, we obtain a trivial identity, however in terms of the coefficients of the recurrence relation we have
\begin{align*}
\left(v'(L)\right)_{0}(n)
&=
B(n)(C(n)+C(n+1))+(C(n)+C(n+1)+B(n)^2+2t)B(n),
\end{align*}
which implies the identity
\begin{equation*}
\begin{aligned}
\left[C(n),A\right]
&=
C(n)B(n-1)(C(n-1)+C(n))\\
&+C(n)(C(n-1)+C(n)+B(n-1)^2+2t)B(n-1)\\
&-B(n)(C(n)+C(n+1))C(n)\\
&-(C(n)+C(n+1)+B(n)^2+2t)B(n)C(n).
\end{aligned}
\end{equation*}
On the other hand, we can sum the second identity from $0$ to $n-1$, to obtain
\begin{equation*}
\sum_{k=0}^{n-1}[B(k),A]
=
n+ \left(v'(L)\right)_{-1}(0)- \left(v'(L)\right)_{-1}(n)
=
n- \left(v'(L)\right)_{-1}(n)
= n-A_{-1}(n),
\end{equation*}
using \eqref{Ajv} as well as the fact that $C(0)=0$. Therefore, we obtain
\begin{equation*}
\begin{aligned}
\sum_{k=0}^{n-1}[B(k),A]
&=n-
4 \Bigl(C(n) C( n-1)+ C(n)^2+ C(n+1) C(n)+ B(n)^2 C(n)
\\
&\qquad+ B(n)C(n)C( n-1)+ C(n)B(n-1)^2\Bigr)
-2 t C(n)
\end{aligned}
\end{equation*}

If $A=0$ then the weight is scalar and even, therefore $B(n)=0$ and the recurrence coefficients commute. In this case, the previous equality reduces to 
\[
n
=
4C(n)(C( n-1)+C(n)+ C(n+1)+2t),
\]
which is the discrete Painlev\'{e} I equation, see e.g. \cite[\S 1.2.2]{vanAssche}.
\end{example}

\section{Ladder relations for exponential weights}
\label{sec:exponential-weights}
In this section we investigate the existence of lowering and raising relations for a class of matrix valued weights of exponential type on the real line. This is an important example that fits into the general theory presented in Section \ref{sec:pre}.

\begin{definition}\label{def:exp}
An $N\times N$ weight matrix supported on $\mathbb{R}$ is
called an \textit{exponential type} weight if it is of the form
\begin{equation}
\label{eq:expweight}
W(x)=e^{-v(x)}e^{xA} e^{xA^\ast}, \qquad v(x)=x^k+v_{k-1}x^{k-1}+\cdots+v_0,
\end{equation}
where $v$ is a scalar polynomial, $k$ is even and $A$ is a constant
matrix. 
\end{definition}

\begin{remark} \label{rmk:similar} 
If we consider more general weights 
\begin{equation}\label{WATL}
W_{(A,T,L)}(x) = e^{-v(x)}L e^{xA} T e^{xA^\ast} L^\ast,
\end{equation}
for some constant positive definite matrix $T$ and constant invertible matrix $L$, then using the Cholesky decomposition $T=KK^\ast$, we get
the following similarity:
$$
W_{(A,T,L)}(x) = e^{-v(x)} L K e^{x K^{-1}AK} e^{x (K^{-1}AK)^\ast} (LK)^\ast
= LK W_{(K^{-1}AK,I,I)}(x)(LK)^\ast.
$$
If we denote by  $P(x,n;A,T,L)$ the monic MVOPs with respect to $W_{(A,T,L)}(x)$, then 
$$
P(x,n;A,T,L)=LK P(x,n;K^{-1}AK,I,I) (LK)^{-1}.
$$
Because of this, the weight $W_{(A,T,L)}(x)$ will be called exponential weight as well, and 
we conclude that it suffices to consider only weights of the form \eqref{eq:expweight}.

\end{remark}

If we start with the matrix valued differential operator used before,
 \begin{equation}
\label{eq:def-operatorD}
\D=\partial_x+A,
 \end{equation}
we can prove the form of the adjoint in accordance with Section \ref{sec:general-weights}:
\begin{proposition}
\label{prop:adj}
The adjoint of the differential operator \eqref{eq:def-operatorD} with respect to $W$ is
 $$\D^{\dagger}=-\partial_x-A+v'(x) =  -\D + v'(x).$$
\end{proposition}
\begin{proof}
For polynomials $P,Q\in \MN[x]$, we have from \eqref{eq:left_inner_prod} and \eqref{eq:def-operatorD} that
 \begin{equation*}
 \langle P\cdot \D, Q\rangle = \langle P',Q \rangle + \langle PA,Q \rangle.
 \end{equation*}
Integrating $\langle P',Q \rangle$ by parts  and using that $W$ is invertible and self-adjoint, we obtain
\begin{align}
\label{eq:propDadj-eq2}
\langle P',Q \rangle &= - \langle P,Q' \rangle - \langle P,QW'W^{-1} \rangle = \langle P, -Q' - QW'W^{-1} \rangle.
 \end{align}
Observe that the boundary terms on the right hand side of  \eqref{eq:propDadj-eq2} vanish because of the exponential decay of the matrix weight at $\pm \infty$. Taking into account that $\langle PA, Q \rangle = \langle P, Q WA^\ast W^{-1} \rangle$, we have
$$\langle P\cdot \D, Q\rangle = \langle P, -Q' - QW'W^{-1} + Q WA^\ast W^{-1} \rangle.$$
Therefore $\D^\dagger= -\partial_x - W'(x)W(x)^{-1} +W(x)A^* W(x)^{-1}$, which is the essentially the formal $W$-adjoint used in \cite{Casper2} for
our weight. Putting back the explicit expression for $W$ completes the proof of the proposition.
\end{proof}

\begin{corollary}
	Let $W$ be a matrix weight as in \eqref{eq:expweight}, then $\D = \partial_x + A \in \mathcal{F}_R(P)$.
\end{corollary}
\begin{proof}
	This follows directly from Proposition \ref{prop:adj} and \cite[Theorem 3.7]{Casper2}.
\end{proof}

\begin{remark} 
If a weight $\widetilde{W}(x)$ as in Theorem \ref{lem:commute} is differentiable, then the equation $\mathcal{D}^\dagger=-\mathcal{D}+v'(x)$ implies that $\widetilde{W}(x)$ solves the following first order linear matrix differential equation:
\[
\widetilde{W}'(x)=(A-v'(x))\widetilde{W}(x)+\widetilde{W}(x)A^\ast,
\]
and it follows that $\widetilde{W}(x)$ is an exponential type weight.
\end{remark}

\begin{remark} 
Following the lines of \cite{BFM1}, it is clear that a weight of the form \eqref{eq:expweight} can be factorized in a very natural way as $W(x)=W^{\rm L}(x)W^{\rm R}(x)$, with
\begin{equation}
W^{\rm L}(x)=e^{-\frac{v(x)}{2}}e^{xA} , \qquad
W^{\rm R}(x)=e^{-\frac{v(x)}{2}}e^{xA^*}.
\end{equation}
We note that this factorization is not unique. This leads to the following logarithmic derivatives:
\begin{equation}\label{hLhR}
\begin{aligned}
h^{\rm L}(x)&=\left(W^{\rm L}(x)\right)'W^{\rm L}(x)^{-1}=-\frac{v'(x)}{2}+A , \\
h^{\rm R}(x)&=\left(W^{\rm R}(x)\right)'W^{\rm R}(x)^{-1}=-\frac{v'(x)}{2}+A^*.
\end{aligned} 
\end{equation}
This would allow us to rederive some of the results obtained later in the paper, using the Riemann--Hilbert formulation. However, in this paper we adopt a different approach, based on the use of a specific differential operator $\mathcal{D}$ and its adjoint, which is very well suited to this type of exponential weights and bypasses the lengthy calculations needed in the Riemann--Hilbert formulation of more general cases of MVOPs.

In particular, if we let $v(x)=x^2+tx$ as in Example \ref{ex:Hermiteweight}, and we set $t=0$, then \eqref{hLhR} gives $h^{\rm L}(x)=A-x$, $h^{\rm R}(x)=A^*-x$. We identify $A^{\rm L}=A^{\rm R}=-I$ and $B^{\rm L}=A$, $B^{\rm R}=A^*$ in the notation used in \cite[Section 6.1]{BFM1}. Then the general equations at the end of that section simplify considerably and are fully consistent with \eqref{eq:res1}.
\end{remark}

\section{Hermite-type weights}
\label{sec:hermite}

If we replace $v(x)=x^2+tx$ in \eqref{eq:expweight}, then we recover the results in Example \ref{ex:Hermiteweight}. We say that this is a Hermite-type weight. In this section we investigate further properties of these Hermite-type matrix weights. First we will show that an arbitrary matrix weight having operators $\D$ and $\D^\dagger$ as in Theorem \ref{lem:commute} and with a specific moment of order zero is equivalent to a Hermite-type weight, in a sense that we specify later on. We also investigate a particular case of the matrix $A$ that leads to the Harmonic oscillator algebra and has a link with a quantum mechanical composite system.

\subsection{Characterization of Hermite-type weights with a ladder relation} \label{sec:retrieve}
The proof of this characterization follows the lines of the main result in \cite{BonanNevai}, where the authors discuss a scalar Freud weight. First we present a recursive equation for the norms $\mathcal{\widetilde H}(n)$:

\begin{lemma}
	\label{lem:Hermite_Wtilde}
	Let $\widetilde{W}$ be a  matrix weight, supported on $\mathbb{R}$, and let $(\widetilde  P(x,n))_n$ be the sequence of monic orthogonal polynomials. Let $A$ be a matrix such that
	\begin{equation}\label{eq:DDHermite}
		\mathcal{D} = \partial_x + A \in \mathcal{F}_R(P), \qquad \mathcal{D}^\dagger = -\mathcal{D}+2x+t.
	\end{equation}
Then the square norms, 
\begin{equation*}
\mathcal{\widetilde H}(n)=\int_{-\infty}^\infty \widetilde{P}(n,x) \widetilde W(x) \widetilde{P}(n,x)^*dx, \qquad n\geq 0
\end{equation*}
satisfy the following recursion:
\begin{multline}\label{eq:recurHn}
\widetilde \cH(n+1)
=
\frac{1}{2} \widetilde\cH(n)
+
\widetilde\cH(n)\widetilde\cH(n-1)^{-1} \widetilde\cH(n)\\
-
\frac{1}{4}\widetilde\cH(n)A^\ast\widetilde\cH(n)^{-1}A\widetilde\cH(n)
+
\frac{1}{4}A\widetilde\cH(n)A^\ast,
\end{multline}
for $n>0$. Moreover, $\widetilde C(0)=0$ gives the initial condition
\begin{equation*}
\widetilde\cH(1)
=
\frac{1}{2}\widetilde\cH(0)
-
\frac{1}{4}\widetilde\cH(0)A^\ast\widetilde\cH(0)^{-1}A\widetilde\cH(0)
+
\frac{1}{4}A\widetilde\cH(0)A^\ast.
\end{equation*}

\end{lemma}
\begin{proof}
We write the string equations \eqref{eq:ladder-in-theorem} in terms of the squared norms. Using the isomorphism $\varphi^{-1}$ and \eqref{eq:wehave} we verify that $2L=M+M^\dagger-t$ and using \eqref{eq:M-Mdagger} we obtain
\begin{equation}
\label{eq:Bn-as-norms}
2\widetilde B(n)=A+\mathcal{\widetilde H}(n)A^\ast \mathcal{\widetilde H}(n)^{-1} -t.
\end{equation}
If we replace \eqref{eq:Bn-as-norms} in the first identity of \eqref{eq:res1},
we obtain
\begin{equation}
\label{eq:Cn-as-norms}
\widetilde C(n)-\widetilde C(n+1)
=
\frac{1}{4}\left[ \widetilde \cH(n)A^\ast\widetilde \cH(n)^{-1},A \right]-\frac{1}{2}.
\end{equation}
Since $\widetilde C(n)=\widetilde \cH(n)\widetilde \cH(n-1)^{-1}$ for all $n>0$, we get \eqref{eq:recurHn}.
\end{proof}

Using this result, we can prove the following characterization of Hermite-type weights in terms of the $\mathcal{H}(n)$:
\begin{theorem}
Let $\widetilde{W}$ be a  matrix weight as in Lemma \ref{lem:Hermite_Wtilde} and assume that
\begin{equation*}
\mathcal{\widetilde H}(0)=\int_{-\infty}^\infty e^{-x^2-xt}e^{xA}e^{xA^\ast}dx.
\end{equation*}
Then there is an invertible constant matrix $K$ such that $W(x) = K \widetilde{W}(x) K^\ast$, where
\begin{equation}
\label{eq:HermiteInTheorem}
W(x) = e^{-x^2-xt}e^{xA}e^{xA^\ast},
\end{equation}
almost everywhere with respect to Lebesgue measure on $\mathbb{R}$.
\end{theorem}

\begin{proof}
In this proof, we first show that the string equations determine uniquely, up to the zeroth moment, the coefficients of the recurrence relation for the monic MVOPs with respect to $\widetilde W$. Then the theorem follows by showing that the matrix weight \eqref{eq:HermiteInTheorem} corresponds to a determinate moment problem, in the sense of \cite{Berg_matrix}.

From the previous lemma, the squared norms $\widetilde \cH(n)$ for $n>0$ are completely determined by the choice of $\widetilde \cH(0)$. Furthermore, using the identities $\widetilde C(n) = \widetilde \cH(n) \widetilde \cH(n-1)^{-1}$ and $2\widetilde B(n)=A+\widetilde \cH(n)A^\ast \widetilde \cH(n)^{-1}$, we find that the coefficients of the recurrence relation and, thus, the monic orthogonal polynomials are completely determined as well.

Finally we need to prove that the moment problem for the Hermite weight $W$ has a unique solution. By \cite[Theorem 3.6]{Berg_matrix}, it suffices to show that the diagonal entries of the matrix valued measure $W(x)dx$ are determinate. We follow the approach given by Freud in  \cite[Theorem 5.1 and 5.2]{Freud_OP}: We observe that
\begin{multline*}
\left|\int_{-\infty}^{\infty}e^{\beta |x|}W(x)_{i,i}dx \right|\leq \left\|\int_{-\infty}^{\infty}e^{\beta |x|}e^{-x^2-xt}e^{xA}e^{xA^\ast}dx
\right\|_1 \\
\leq
\int_{-\infty}^{\infty}e^{\beta |x|}e^{-x^2-xt}e^{x\|A\|_1} e^{x\|A^\ast\|_1}dx 
\leq 
M<\infty,
\end{multline*}
for any $\beta>0$. Therefore by \cite[Theorem 5.2]{Freud_OP} the diagonal measures $W(x)_{i,i}dx$ are determinate and so is $W$.
\end{proof}

\subsection{The harmonic oscillator algebra}
\label{sec:harmonic}

We consider a weight related to the one studied in \cite{IKR2}, of the form
\begin{equation} \label{eq:ourweight}
W(x) = e^{-x^2} L(x) L(x)^\ast,
\qquad
L(x)_{j,k}
=
\begin{cases}
\displaystyle \frac{H_{j-k}(x)}{(j-k)!}\frac{\alpha_j}{\alpha_k}
& j \geq k, \\[2mm]
0 & j < k
\end{cases}
\end{equation}
where $\alpha$ is vector of 
positive parameters $(\alpha_j)_{j=1}^N$ and  $H_n(x)$ are the standard scalar Hermite polynomials.

This weight matrix $W(x)$ in \eqref{eq:ourweight} is of exponential type,
as defined in Definition \ref{def:exp}; this result follows from the fact that $L(x)$ satisfies the matrix ODE
$$
\frac{d}{dx} L(x)
=
L(x) A
=
A L(x),
\qquad
A_{j,k}
=
\begin{cases}
\frac{2 \alpha_{k+1}}{\alpha_{k}} & k = j-1 \\
0 & \textrm{else}
\end{cases}.
$$
Then, we can write $L(x) = L(0) e^{x A}$,
and we can recast \eqref{eq:ourweight} in the form of an exponential type weight as given in \eqref{WATL}.
\begin{remark}
The exponential weight considered in \eqref{eq:ourweight} is taken in such a way that the
$0$-th  square norm (of the monic MVOP) is diagonal; moreover, using the recurrence relation for the norms in Theorem \ref{lem:Hermite_Wtilde}, the special structure of the matrix $A$ implies that all square norms are diagonal as well.
\end{remark}

\begin{lemma}\label{lem:Hn}
The squared norms $\cH(n)$ of the monic MVOPs with respect to \eqref{eq:ourweight} and the three term recursion coefficient $C(n)$ are diagonal for all $n \in \mathbb{N}_{0}$.
\end{lemma}
\begin{proof}
We first show that $\cH(0)_{j,j}$ is a diagonal matrix:
\begin{align*}
\left( \cH (0) \right)_{j,k}
&=
\sum_{\ell =1}^N
\int_{-\infty}^\infty
e^{-x^2}
L(x)_{j,\ell}(L(x)^\ast)_{\ell,k} dx
\\
&=
\sum_{\ell =1}^{\min(j,k)}
\int_{-\infty}^\infty
e^{-x^2}
\frac{H_{j-\ell}(x)}{(j-\ell)!}\frac{\alpha_j}{\alpha_\ell}
\frac{H_{k-\ell}(x)}{(k-\ell)!}\frac{\alpha_k}{\alpha_\ell}
dx
\\
&=
2^j \alpha_j^2
\delta_{j,k}
\sum_{\ell=1}^j
\frac{2^{-\ell}}{(j-\ell)!}\frac{1}{\alpha_\ell^2}
\end{align*}
Next, if we denote the square norms in \eqref{eq:recurHn} with $t=0$ as
$\widetilde \cH(n)$, and by $\cH(n)$ the square norms
of the monic polynomials with respect to the weight \eqref{eq:ourweight}, then 
\begin{equation}\label{eq:HnHntilde}
L(0) \widetilde \cH(n) L(0)^\ast = \cH(n).
\end{equation}
Therefore, the $\cH(n)$ satisfy the same second order recursion \eqref{eq:recurHn}. Finally, since our matrix $A$ in this section only has non-zero entries
on the first subdiagonal, it follows that all the terms in the
recursion \eqref{eq:recurHn} are diagonal. The result for $C(n)$ follows from the formula $ C(n)= \cH(n)\cH(n-1)^{-1}$.
\end{proof}

The next proposition is a direct consequence of \cite[Proposition 3.5]{IKR2}:
\begin{proposition}\label{prop:D_acting_on_P}
The monic MVOPs $P(x,n)$ that correspond to \eqref{eq:ourweight}
satisfy
\begin{equation}
\label{eq:PeigenfunctionD}
(P\cdot D)(x,n) = \Gamma(n)  P(x,n),
\end{equation}
with the self-adjoint second order differential operator
\begin{equation}\label{eq:DHermite}
D =
-\frac12 \partial_x^2 
+ \partial_x (xI-A)
+ J, \qquad J \vcentcolon = \mathrm{diag}(1,2,\dots, N),
\end{equation}
that acts from the right and with eigenvalues $\Gamma(n) = nI+J.$
\end{proposition}

\begin{proposition}
The identity operator and the differential operators $\mathcal{D}$, $\mathcal{D}^\dagger$, $D$ given in \eqref{eq:DDHermite} and \eqref{eq:DHermite} generate a four dimensional Lie algebra called the \emph{harmonic oscillator algebra}, that we denote by $\mathfrak{h}$. 
\end{proposition}
\begin{proof}
The differential operators introduced in this section satisfy the
following commutation relations
\begin{equation}
 \label{eq:commutation-relations}
 [D,  \D] =   \D, \qquad [D, \D^\dagger] = -  \D^\dagger, \qquad [\D,\D^\dagger]= -2,
\end{equation}
that follow from direct calculation. Note that they act from the right. With the identification
$$ 
\D \longleftrightarrow \sqrt{2} \mathcal{J}^+, 
\qquad  \D^\dagger \longleftrightarrow \sqrt{2} \mathcal{J}^-,
\qquad  D \longleftrightarrow \mathcal{J}^3,
\qquad  I\longleftrightarrow \mathcal{E},$$
we find that $\mathfrak{h}$ is isomorphic to the four dimensional Lie algebra $\mathcal{G}(a,b)$ given in \cite[\S 2.5]{miller_1} with parameters $a=0$ and $b=1$, see also \cite[Chapter 10, (1.1)]{MR0338286}. 
\end{proof}
\begin{remark}
Since $D$, $ \D$, $ \D^\dagger$ are elements of $\mathcal{F}_R(P)$, 
we have that $\mathfrak{h}\subset \mathcal{F}_R(P)$. 
Moreover the Lie algebra $\mathfrak{g}$ in 
Theorem \ref{thm:LieAlgebra} is a three dimensional ideal 
of the Lie algebra $\mathfrak{h}$. 
The isomorphism $\varphi$ immediately gives an 
isomorphic subalgebra 
$\varphi^{-1}(\mathfrak{h}) \subset \mathcal{F}_L(P)$.
\end{remark}

The Casimir operator of this Lie algebra $\mathfrak{h}$ is given by
\begin{equation}\label{eq:Casimir}
\mathcal{C}
= 
 D - \frac12 \D^\dagger  \D
=
J - x A + \frac12 A^2,
\end{equation}
using the explicit expressions for $D$, $\mathcal{D}$ and $\mathcal{D}^\dagger$. It can be easily seen from the commutation relations \eqref{eq:commutation-relations} that $\mathcal{C}$ commutes with $\widetilde \D, \widetilde \D^\dagger$ and $D$. The Casimir operator is useful in order to derive another differential identity for the MVOPs, that we present below.

\begin{lemma}
The Casimir operator is self-adjoint and it acts on the monic MVOPs as
$$
(P\cdot \mathcal{C})(x,n)
=
(\varphi^{-1}(\mathcal{C})\cdot P)(x,n),
$$
with
\begin{multline*}
\varphi^{-1}(\mathcal{C}) =
-A \delta
+
\left(n I+J - 2 C(n)-A B(n)+ \frac12 A^2 \right)\\
+
\left( C(n) A - 2 C(n) B(n-1)  \right)\delta^{-1}.
\end{multline*}
\end{lemma}
\begin{proof} 
The fact that the Casimir operator is self-adjoint follows directly from the expression 
$\mathcal{C} = D - \frac12 \D^\dagger  \D$.  Now applying the Lie algebra isomorphism,
\begin{align*}
\varphi^{-1}(\mathcal{C})
&=
\varphi^{-1}(D)
- \frac12
\varphi^{-1}(\mathcal{\D^\dagger})\varphi^{-1}(\mathcal{\D}) = \Gamma(n)
- \frac12 M^\dagger M.
\end{align*}
Now the explicit expression of $\varphi^{-1}(\mathcal{C})$ follows by replacing the explicit expressions of $\Gamma(n)$ and $M, M^\dagger$ from Example \ref{ex:Hermiteweight}, with $t=0$ and noting that \eqref{eq:HnHntilde} respect the form of the the ladder relations.
\end{proof}

Observe that the Casimir operator $\varphi^{-1}(\mathcal{C})$ of the Lie algebra $\varphi^{-1}(\mathfrak{h})$ is a second order difference operator having the sequence of monic MVOPs as eigenfunctions with a non-diagonal eigenvalue acting on $P(x,n)$ from the right. Therefore, $\varphi^{-1}(\mathcal{C})$ is an element of the left bispectral algebra $\mathcal{B}_L(P)$ given in \eqref{eq:definition-bispectral}.

The fact that the operator $\mathcal{C}$ commutes with $\widetilde \D$, $\widetilde \D^\dagger$ and $D$ is translated via the isomorphism $\varphi$ into the following relations
$$[\varphi^{-1}(\mathcal{C}),\widetilde M]=0,\qquad
[\varphi^{-1}(\mathcal{C}), \widetilde M^\dagger]=0, \qquad
[\varphi^{-1}(\mathcal{C}), \Gamma(n)]=0.$$

\begin{remark}
The scalar Hermite polynomials have a well-known application as part of the solution
to the quantum harmonic oscillator, so it is natural to seek an analogous link in our matrix valued case. 
In the next section we also link the differential
equation satisfied by our polynomials to a Schr\"{o}dinger equation, but it is in fact several
copies of the same Schr\"{o}dinger equation.
\end{remark}

\subsection{Computation of Hermite MVOPs}

The nonlinear recurrence for the norms $\mathcal{H}(n)$ that we obtained before is important from a computational point of view as well. If we want to compute Hermite MVOPs, \eqref{eq:recurHn} is a very convenient alternative to Gram--Schmidt orthogonalization applied to the canonical basis; together with \eqref{eq:Bn-as-norms} and \eqref{eq:Cn-as-norms} to calculate the recurrence coefficients, this gives a double recursion (first \eqref{eq:recurHn} and then the recurrence relation) that can then be used to compute $P(x,n)$.

Two drawbacks of this approach are that we need to calculate the inverse $\mathcal{H}(n)^{-1}$ at each step in order to use \eqref{eq:recurHn}, and also that the whole procedure can be slow because of the combination of those two matrix recursions, which generally involve full matrices. For this reason we show next an alternative for the class of weights that we consider in this section: we can replace the matrix recurrence relation by a scalar
recursion for some coefficients $\xi(n,N,k)$, and additionally the
square norms $\mathcal{H}(n)$ can be easily made diagonal, significantly reducing
the computation of the inverses that occur in \eqref{eq:recurHn}.

This result is related to the results in \cite{IKR2}, where the
MVOPs are calculated explicitly but one needs to impose 
certain additional constraints on parameters in the weight matrix.

We denote by $P(x,n)$ the monic MVOPs with respect to \eqref{eq:ourweight}, and we also define the following auxiliary functions, which will be useful later:
\begin{equation}\label{eq:Qxn}
Q(x,n) = P(x,n) \Phi(x), \qquad \Phi(x) = e^{-x^2/2}L(x).
\end{equation}

\begin{lemma}\label{lem:niceL}
The differential operators $\mathcal{D}_Q$, $\mathcal{C}_Q$ and $D_Q$ defined by
\begin{equation}\label{eq:DQCQ}
\mathcal{D}_Q \vcentcolon= \Phi(x)^{-1}\mathcal{D}\Phi(x) , \qquad
\mathcal{C}_Q \vcentcolon= \Phi(x)^{-1}\mathcal{C}\Phi(x), \qquad
D_Q \vcentcolon= \Phi(x)^{-1}D\Phi(x),
\end{equation}
are given explicitly as follows:
\begin{equation}
\mathcal{D}_Q = \partial_x + x, \qquad
\mathcal{C}_Q = J, \qquad
D_Q =
\frac12 \left( - \partial_x^2 + x^2 - 1 \right)I + J.
\end{equation}
\end{lemma}
\begin{proof}
Since $L'(x)=AL(x),$ we immediately obtain $(L(x)^{-1})'= -AL(x)^{-1}$ and therefore $
(\Phi(x)^{-1})' = (x-A)\Phi(x)^{-1}.$ Using this expression we get
$$
\D_Q = \partial_x + x.
$$

We recall from \cite[Lemma 3.4]{IKR2} that $L(x)^{-1}JL(x) = J-\frac12 A^2+xA$. Now by direct calculation we get the following
equations:
\begin{align*}
[J,\Phi(x)] =xA\Phi(x)-\frac12 A^2 \Phi(x), \qquad
\Phi(x)^{-1}J\Phi(x) = J +xA-\frac12 A^2.
\end{align*}
These imply that $\mathcal{C}_Q=J$. For the second order operator we conjugate the
terms separately.
\begin{align*}
-\frac 12\Phi(x)^{-1}\partial_x^2 \Phi(x)
&=
-\frac12 \partial_x^2 -\partial_x (x -A)
-
\left(\frac12 x^2 -xA +\frac12 A^2 +\frac12  \right),
\\
\Phi(x)^{-1}\partial_x \Phi(x) (xI-A)
&=
\partial_x (x-A) + (x^2 - 2xA + A^2).
\end{align*}
When combined, they lead to the desired expression.
\end{proof}

\begin{proposition}
The matrix elements of $Q(x,n)$ are multiples of scalar Hermite functions
\begin{equation}\label{eq:Qform}
Q(x,n)_{j,k}
=
\xi(n,j,k) H_{n+j-k}(x)e^{-x^2/2},
\end{equation}
for $n+j-k \geq 0$ and equal to $0$ otherwise.
\end{proposition}

\begin{proof}
Equation \eqref{eq:DQCQ} implies that $Q(x,n)_{j,k}$ satisfies
the Schr\H{o}dinger equation for the quantum harmonic oscillator
$$
-\frac12 Q(x,n)_{j,k}'' + \frac12 x^2 Q(x,n)_{j,k}
= \left( n +j -k +\frac12 \right) Q(x,n)_{j,k}.
$$
So the functions $Q(x,n)_{j,k}$ should each be linear combinations of the bounded and unbounded
solutions to the above ODE. But since we know $Q(x,n)$ is a matrix polynomial
$P(x,n)L(x)$ multiplied by $e^{-x^2/2}$, the entries can only be
equal to the bounded solution which is the Hermite 
function $H_{n+j-k}(x)e^{-x^2/2}$ when $n+j-k \geq 0$ and the zero function otherwise.
\end{proof}

\begin{proposition}
The constants in \eqref{eq:Qform} have the special values:
\begin{align}
\label{eq:bcxi}
\xi(n,N,k) &=  \frac{2^{-n}}{(N-k)!}\frac{\alpha_N}{\alpha_k}, \\
 \quad \xi(0,j,k) &= \frac{1}{(j-k)!}\frac{\alpha_j}{\alpha_k}, \qquad \text{for } j \geq k, \quad \xi(0,j,k) = 0 \qquad \text{for } j<k.
\end{align}
\end{proposition}

\begin{proof}
For $n=0$ we simply have $Q(x,0)=L(x)e^{-x^2/2}$, so we can directly
read off the desired expression from \eqref{eq:ourweight}.

We can determine the constants for $j=N$ by using the monicity of $P(x,n) = x^n I + \sum_{s=1}^{n} p_s(n) x^{n-s}$, comparing powers in $x$ 
and considering specific entries. 
Writing out the matrix exponential $L(x)=L(0)e^{xA}$ which is a polynomial of degree $N-1$
we have (for $k\leq N$)
\begin{align*}
e^{x^2/2}Q(x,n) &= x^{n+N-1}\left( \frac{1}{(N-1)!} L(0)A^{N-1} \right) \\
&+  x^{n+N-2}\left( \frac{1}{(N-1)!} p_1(n) L(0) A^{N-1} + \frac{1}{(N-2)!} L(0) A^{N-2} \right) \\ 
&\,\,\vdots \\
&+  x^{n+N-k}\left( \frac{1}{(N-r)!} p_{q}(n) L(0) A^{N-r} + \dots + \frac{1}{(N-k)!} L(0) A^{N-k} \right)  \\
&\,\,\vdots
\end{align*}
where $q = \min(k-1,n)$ and $r=k-q$.
For each power we only know one term explicitly: the right most one because
it is without $p_s(n)$. 
But the other terms only have non-zero entries in the first $k-1$ columns\footnote{This relies on the
fact that $A^r$ only has a non-zero $r$-th subdiagonal and $L(0)$ is lower triangular.} whereas the last term has exactly one non-zero 
entry in column $k$
\begin{equation*}
\left( L(0) A^{N-k} \right)_{i,k} = \delta_{i,N} L(0)_{N,N}(A^{N-k})_{N,k} = \prod_{j=k+1}^N \frac{2\alpha_{j}}{\alpha_{j-1}} = 2^{N-k}\frac{\alpha_N}{\alpha_k}.
\end{equation*}
Comparing this to the leading coefficient in \eqref{eq:Qform}
$$
e^{x^2/2}Q(x,n)_{N,k} = \xi(n,N,k) 2^{n+N-k} x^{n+N-k} +\dots
$$
we conclude that 
$$
\xi(n,N,k) =  \frac{2^{-n}}{(N-k)!}\frac{\alpha_N}{\alpha_k}.
$$ 
\end{proof}

\begin{proposition}\label{prop:casiQ}
The auxiliary function satisfies the following first order matrix ODE
with coefficients that depend on the square norms:
\begin{align*}
&Q(x,n)J\\
&=
\left( nI +J -\frac12 x A -\frac12 \cH(n)A^\ast \cH(n)^{-1}(x-A)-2 \cH(n)\cH(n-1)^{-1}\right)Q(x,n)
\\
&+ \frac12( A - \cH(n)A^\ast \cH(n)^{-1} )Q'(x,n),
\qquad n \geq 1
\end{align*}
and for $n=0$
\begin{equation*}
\begin{aligned}
Q(x,0)J
&=
\left( J -\frac12 x A -\frac12 \cH(0)A^\ast \cH(0)^{-1}(x-A)\right)Q(x,0)
\\
&\qquad + \frac12( A - \cH(0)A^\ast \cH(0)^{-1} )Q'(x,0).
\end{aligned}
\end{equation*}
\end{proposition}
\begin{proof}

We start off with the equation that we got from the Casimir
and eliminate any $Q(x,m)$ with $m\neq n$ by using the three 
term recurrence in the first step and the lowering relation
$M\cdot Q = Q \cdot \D_Q$
in the second.
\begin{align*}
(Q\cdot (\mathcal{C}_Q) )(x,n)
&=
- A Q(x,n+1)
+
\left(n I+J - 2 C(n) -A B(n) + \frac12 A^2 \right)Q(x,n)
\\
&+
\biggl( C(n) A - 2 C(n) B(n-1)  \biggr)Q(x,n-1)\\
&=
\left(n I+J - x A + \frac12 A^2 -2 C(n)\right)Q(x,n)
\\
&+
\biggl( C(n) A+A C(n) - 2 C(n) B(n-1)  \biggr)Q(x,n-1)\\
&=
\left(n I+J - x A + \frac12 A^2 -2 C(n)\right)Q(x,n)
\\
&+
\frac12 \biggl( C(n) A+A C(n) - 2 C(n) B(n-1)  \biggr)C(n)^{-1}(x-A)Q(x,n) 
\\
&+ \frac12 \biggl( C(n) A+A C(n) - 2 C(n) B(n-1)  \biggr)C(n)^{-1} Q'(x,n)
\\
&=
\left(n I+J - B(n) x  -2 C(n) +\frac12 \cH(n)A^\ast \cH(n)^{-1} A \right)Q(x,n)
\\
&+ (A- B(n)) Q'(x,n)
\end{align*}
In the final steps we clean up the expression using
the expressions for $B(n)$ and $C(n)$ in terms of the square norms.
If we leave everything in terms of the square norms we get the desired expression.
\end{proof}

\begin{theorem}\label{thm:xirec}
We have the following (scalar) recursion for $2 \leq j \leq N-1$ and
$n\geq 1$
\begin{align*}
&\xi(n,j-1,k)\\
&=
	\left(
	\frac{\alpha_{j-1}}{\alpha_j}(n+j-k)
	+2 \frac{\alpha_{j-1}\alpha_{j+1}^2}{\alpha_j^3}\frac{\cH(n)_{jj}}{\cH(n)_{j+1,j+1}}
	-2 \frac{\alpha_{j-1}}{\alpha_j} \frac{\cH(n)_{jj}}{\cH(n-1)_{jj}}  \right)\xi(n,j,k)
	\\
&- 2	(n+j-k+1) 
 \frac{\alpha_{j-1}\alpha_{j+1}}{\alpha_j^2} \frac{\cH(n)_{jj}}{\cH(n)_{j+1,j+1}} \xi(n,j+1,k),
\end{align*}
and for $j=N$ we get
\begin{equation*}
\begin{aligned}
\xi(n,N-1,k)
&=
	\left(
	\frac{\alpha_{N-1}}{\alpha_N}(n+N-k)
	-2 \frac{\alpha_{N-1}}{\alpha_N} \frac{\cH(n)_{NN}}{\cH(n-1)_{NN}}  \right)\xi(n,N,k).
\end{aligned}
\end{equation*}
\end{theorem}
\begin{proof}
The proof follows from looking at the entries of result of Proposition
\ref{prop:casiQ} when we will in $x=0$.
\end{proof}

In conclusion, to compute $P(x,n)$ for a given $n$ one must
\begin{enumerate}
\item Compute $\cH(m)$ for $1 \leq m \leq n$
using \eqref{eq:recurHn} and the explicit form of $\cH(0)$ from 
the proof of Lemma \ref{lem:Hn}.
\item Compute the $\xi(n,j,k)$ using the recursion of Theorem \ref{thm:xirec}
with the boundary condition from \eqref{eq:bcxi}.
\item Then we have $Q(x,n)$ and so $P(x,n)=Q(x,n)L(x)^{-1}e^{x^2/2},$
where we take an explicit expression for $L(x)^{-1}$
from \cite[Proposition 3.1]{IKR2}.
\end{enumerate}

\section{Matrix valued Pearson equations and ladder relations}
\label{sec:Pearson}
In this section we specialize the matrix $A$ in \eqref{eq:ourweight} in such a way that it  satisfies a Pearson equation. This gives rise to new ladder relations and allows to obtain more explicit results for the case of Hermite and Freud-type weights.

\begin{remark}
A word should be said about the name Pearson
equation in the context of \textit{matrix} weights. 
In papers such as \cite{DG_2005b, CanteroMV2005, CanteroMV2007, IKR2} the name Pearson equation is
used for 
an equation for the matrix weight of the form
\begin{equation}\label{eq:Pearsondef}
(W\Phi)'(x) = W(x)\Psi(x),
\end{equation}
with $\Phi$ and $\Psi$ polynomials of degree two and one respectively.
In other sources, such as \cite{BFM1,CM1}, the same
name is used for a more general equation
of the form
$$
W'(x) = A(x)W(x) + W(x)B(x),
$$
with matrix polynomials $A$ and $B$.
\end{remark}

In this section we assume that we have a weight matrix $W$ which satisfies a Pearson equation of the form
\begin{equation}
\label{eq:PearsonVGral}
W'(x)=-W(x)V(x),
\end{equation}
where $V(x)$ is a matrix valued polynomial of degree $k$. The case where $V$ is a polynomial of degree one is studied in detail in \cite{DG_2005b}. Using integration by parts, we prove that there exist matrix valued sequences $M_{-2}(n),\ldots, M_{-k}(n)$ such that the  monic orthogonal polynomials $P(x,n)$ with respect to $W$ satisfy:
\begin{equation}
\label{eq:Pearson-Derivative}
(P\cdot \partial_x) (x,n) = P'(x,n) = n P(x,n-1) + M_{-2}(n) P(x,n-2) + \cdots + M_{-k}(n) P(x,n-k).
\end{equation}
Therefore, $\partial_x \in \mathcal{F}_R(P)$ and
$$\varphi^{-1}(\partial_x) = n \delta^{-1} + M_{-2}(n) \delta^{-2} + \cdots + M_{-k}(n) \delta^{-k}.$$
The operator $\partial_x$ has an adjoint $\partial_x^\dagger$ given by
\begin{equation}\label{eq:dx_adjoint}
\partial_x^\dagger = -\partial_x + V(x)^\ast.
\end{equation}
Moreover, $[\D,\partial_x]=0$ implies that $[\varphi^{-1}(\D),\varphi^{-1}(\partial_x)]=0$.

\subsection{Hermite-type example}
\label{sec:HermitePearson}

We consider the weight of the previous section \eqref{eq:ourweight}
\begin{equation}
\label{eq:weight_factorization-Hermite}
W(x)=e^{-x^2} L(x)L(x)^\ast,
\end{equation}
but with a particular choice for the parameters
$$
\frac{2\alpha_j}{\alpha_{j-1}}
=
\sqrt{(j-1)(N-j+1)}.
$$
In contrast to the previous section, this \textit{is} now in fact a particular case of the family of weights considered in \cite{IKR2}. We
do not seek the entries of the MVOPs as this has been done in \cite{IKR2}, 
but instead we find a simple lowering relation which is not obvious from the
explicit MVOPs.

\begin{proposition}
	\label{prop:PearsonHermite}
	The weight $W$ satisfies the Pearson equation \eqref{eq:PearsonVGral}	where $V(x)$ is the following polynomial of degree two
	$$
	-V(x)
	=
	(L(0)^\ast)^{-1}A L(0)^\ast +A^\ast 
	+2x\left(J +\frac12 (A^\ast)^2-\frac{N+3}{2}\right)
	-x^2A^\ast,
	$$
	where $J$ is the diagonal matrix introduced in Subsection \ref{sec:harmonic}.
\end{proposition}
\begin{proof}
	Using the structure of the matrix weight \eqref{eq:weight_factorization-Hermite} we obtain:
	$$W(x)^{-1}W'(x)= -2x + (L(x)^\ast)^{-1}AL(x)^\ast+A^\ast.$$
	Now observe that
	\begin{equation*}
	\frac{d}{dx}\left((L(x)^\ast)^{-1}AL(x)^\ast\right) 
	=  (
	L(x)^\ast)^{-1} [A,A^\ast] L(x)^\ast.
	\end{equation*}
	On the other hand, since $\frac{1}{4}[A,A^\ast]_{j,k}=\left(\frac{\alpha_j^2}{\alpha_{j-1}^2} - \frac{\alpha_{j+1}^2}{\alpha_{j}^2}\right)\delta_{j,k}$, it follows that
$$
4\left(\frac{\alpha_j^2}{\alpha_{j-1}^2} - \frac{\alpha_{j+1}^2}{\alpha_{j}^2 }\right)
=  \left(2j -(N+1) \right)
\quad \Longrightarrow \quad 
[A,A^\ast] 
= 2 J - (N+1),
$$
	and using the relation $(L(x)^\ast)^{-1}JL(x)^\ast = -xA^\ast + J +\frac12 (A^\ast)^2$ we get 
    $$
    \frac{d}{dx}\left((L(x)^\ast)^{-1}AL(x)^\ast\right) 
    = 
    2(-xA^\ast +J +\frac12 (A^\ast)^2)-(N+1).
    $$
    Integrating with respect to $x$ we complete the proof of the proposition.
\end{proof}
It follows from Example \ref{ex:Hermiteweight} and \eqref{eq:Pearson-Derivative} that
\begin{equation}\label{eq:2lowerings}
\varphi^{-1}(\D) = A +  2C(n) \delta^{-1},\qquad \varphi^{-1}\left(\partial_x\right) = n \delta^{-1} + M_{-2}(n)\delta^{-2},
\end{equation}
for a certain sequence $M_{-2}(n)$.
\begin{proposition}
	\label{prop:Pearson_hermite}
	The sequence in \eqref{eq:2lowerings} is given by
\begin{equation*}
M_{-2}(n) = \cH(n)A^\ast \cH(n-2)^{-1}, \qquad n\geq 2.
\end{equation*}
\end{proposition}
\begin{proof}

We start by collecting the adjoints of the two lowering differential operators
from \eqref{eq:dx_adjoint} and \eqref{DDdagger}
$$
\partial_x^\dagger = -\partial_x +V(x)^\ast, \qquad
\D^\dagger = -\partial_x -A + 2x.
$$
Since $V(x)^\ast$ is of degree 2 with $A$ as a leading coefficient and $P(x,n)$
is monic, we have
$$
(P\cdot (\D^\dagger-\partial_x^\dagger))(x,n)
=
-AP(x,n+2) + \mathrm{(lower\, order\, terms)}
$$
On the other hand we can write its corresponding difference operator using
Corollary \ref{cor:adjoint}
$$
\varphi^{-1}(\D^\dagger-\partial_x^\dagger)
=
-\cH(n)M_{-2}(n+2)^\ast \cH(n+2)^{-1} \delta^2 + \mathrm{(lower\, order\, terms)}.
$$
So we must have
$$
\cH(n)M_{-2}(n+2)^\ast \cH(n+2)^{-1} = A,
$$
which leads to the desired result.
\end{proof}

\begin{corollary}\label{cor:Hrec2}
The squared norms satisfy another second order recursion in $n$
\begin{align*}
2(n+1)\cH(n+1)^{-1}
&-
 2(n+2) \cH(n+2)^{-1}\cH(n+1)\cH(n)^{-1}
\\
&+
\cH(n+2)^{-1}A\cH(n+2)A^\ast \cH(n)^{-1}
-
A^\ast \cH(n)^{-1}A = 0.
\end{align*}
\end{corollary}
\begin{proof}
Using $0=[\varphi^{-1}(\D),\varphi^{-1}(\partial_x)]$, a direct computation gives
\begin{multline*}
0= \left( 2(n-1)C(n)- 2nC(n-1)
- M_{-2}(n)A  +  AM_{-2}(n) \right) \delta^{-2} \\
+ 2(C(n)M_{-2}(n-1)-M_{-2}(n)C(n-2))\delta^{-3}.
\end{multline*}
From the $\delta^{-2}$ term we obtain
\begin{align*}
[M_{-2}(n),A]&=2((n-1)C(n)-nC(n-1)),
\end{align*}
which gives the desired result when written in terms
of the squared norms using Proposition \ref{prop:Pearson_hermite}.
\end{proof}
\begin{remark}
We note that we could reduce the order of the recursion by combining
the result in Corollary \ref{cor:Hrec2} with \eqref{eq:recurHn}.
But from the point of view of the scalar (diagonal) entries $\cH(n)_{jj}$,
this severely raises the order of the recursion in $j$.
\end{remark}

\subsection{Freud-type example}
We consider the weight matrix of Example \ref{ex:Freudweight} again:
\begin{equation}
\label{eq:weight_factorization}
W(x)=e^{-x^4+tx^2} e^{xA}e^{xA^\ast},
\end{equation}
where $A$ is the lower triangular nilpotent matrix defined by $$A_{i,j}=\sqrt{\mu_i}\delta_{i-1,j}, \qquad 6\mu_i=(i-1)(N-i+1)(2N\alpha+2\alpha i+3\beta+\alpha),$$
and $\alpha, \beta$ are real numbers. As in Proposition \ref{prop:Pearson_hermite}, we prove that the weight $W$ satisfies a Pearson equation of the form
	$$W'(x)=-W(x)V(x),$$
where $V(x)$ is a polynomial of degree three. Therefore, there exist sequences $M_{-2}(n)$ and $M_{-3}(n)$ such that
$$P_n\cdot \partial_x = n P_{n-1} + M_{-2}(n)P(x,n-2) + M_{-3}(n)P(x,n-3).$$

\section{Parameter deformation of the weight and multi-time Toda lattice}
\label{sec:Toda}
In this section, we consider an arbitrary matrix weight of the form
\begin{equation}
\label{eq:weight-deformation}
W(x,t)=e^{-v(x,t)} \widetilde{W}(x),
\end{equation}
where  $v(x;t)$ is a polynomial of even degree with positive leading coefficient depending smoothly on a parameter $t\geq 0$. In the following theorem we study the effect of differentiating the recurrence coefficients with respect to $t$, an idea that is natural when one considers orthogonal polynomials in the context of integrable systems.
\begin{theorem} If we denote by  $\, \dot{}\, $  the derivative with respect to $t$, then the recurrence coefficients in \eqref{eq:three_term_monic} satisfy the following deformation equations:
\begin{equation}
\label{eq:lattices}
\begin{aligned}
\dot{B}(n)
&=
\left(\dot{v}(L)\right)_{-1}(n)
-
\left(\dot{v}(L)\right)_{-1}(n+1)
\\
\dot{C}(n)
&=
\left(\dot{v}(L)\right)_{-2}(n)
-
\left(\dot{v}(L)\right)_{-2}(n+1)\\
&+
\left(\dot{v}(L)\right)_{-1}(n)B(n-1)
-B(n) \left(\dot{v}(L)\right)_{-1}(n),
\end{aligned}
\end{equation}
where we use the same notation as in Remark \ref{rmk:notation}.
\end{theorem}
\begin{proof}
Let $P(x,n)$ be the monic orthogonal polynomials with $W(x)$. Taking into account that $\langle P(x,n),P(x,m)\rangle=0$ for $n>m$, we have
\begin{equation*}
\begin{aligned}
0 &= \frac{\partial}{\partial t} \langle P(x,n),P(x,m)\rangle
= \langle \dot{P}(x,n),P(x,m)\rangle
- \langle P(x,n) \cdot \dot{v}(x),P(x,m)\rangle,
\end{aligned}
\end{equation*}
and then we can expand
\[
\begin{aligned}
\dot{P}(x,n)
&=
\sum_{m=0}^{n-1} \left\langle P(x,n)\cdot \dot{v}(x) , P(x,m)\right\rangle\cH(m)^{-1}P(x,m)\\
&=
\sum_{m=0}^{n-1} \left(\dot{v}(L)\right)_{m-n}(n) P(x,m),
\end{aligned}
\]
where we use the notation in Remark \ref{rmk:notation} for $\left(\dot{v}(L)\right)_{k}(n)$. On the other hand, if we differentiate the three-term recurrence relation \eqref{eq:three_term_monic} with respect to $t$, we obtain
\begin{equation}\label{RHS}
\begin{aligned}
x \dot{P}(x,n)&= \dot{P}(x,n+1) + B(n)\dot{P}(x,n) + C(n)\dot{P}(x,n-1)\\
&+ \dot{B}(n) P(x,n) + \dot{C}(n)P(x,n-1) 
\\
&=
\sum_{m=0}^{n}\left(\dot{v}(L)\right)_{m-n-1}(n+1) P(x,m)
+
B(n) \sum_{m=0}^{n-1}\left(\dot{v}(L)\right)_{m-n}(n)  P(x,m)
\\
&+
C(n) \sum_{m=0}^{n-2}\left(\dot{v}(L)\right)_{m-n+1}(n-1) P(x,m)\\
&+ \dot{B}(n) P(x,n) + \dot{C}(n)P(x,n-1),
\end{aligned}
\end{equation}
while on the left hand side we get
\begin{equation}\label{LHS}
\begin{aligned}
x \dot{P}(x,n)
&=
\sum_{m=0}^{n-1}\left(\dot{v}(L)\right)_{m-n}(n) P(x,m+1)\\
&+\sum_{m=0}^{n-1}\left(\dot{v}(L)\right)_{m-n}(n)\left(B(m) P(x,m) 
+C(m) P(x,m-1)\right).
\end{aligned}
\end{equation}

Combining \eqref{RHS} and \eqref{LHS}, isolating the derivatives of the recurrence coefficients,
we get
\[
\begin{split}
&\dot{B}(n) P(x,n) + \dot{C}(n) P(x,n-1)\\
&=\sum_{m=0}^{n-1}\left(\dot{v}(L)\right)_{m-n}(n) P(x,m+1)+ 
\sum_{m=0}^{n-1}\left(\dot{v}(L)\right)_{m-n}(n) B(m) P(x,m) \\
&+\sum_{m=1}^{n-1}\left(\dot{v}(L)\right)_{m-n}(n) C(m) P(x,m-1)
-\sum_{m=0}^{n}\left(\dot{v}(L)\right)_{m-n-1}(n+1) P(x,m)\\
&-B(n) \sum_{m=0}^{n-1}\left(\dot{v}(L)\right)_{m-n}(n)  P(x,m)
-C(n) \sum_{m=0}^{n-2}\left(\dot{v}(L)\right)_{m-n+1}(n-1) P(x,m).
\end{split}
\]
Comparing coefficients of $P(x,n)$ and $P(x,n-1)$ we obtain  \eqref{eq:lattices} for $\dot{B}(n)$ and 
$\dot{C}(n)$.
\end{proof}
\begin{example} \normalfont
 If $\dot{v}(x) = x$, we obtain the non-Abelian Toda lattice
\begin{equation*}
\dot{B}(n) = C(n) - C(n+1),\qquad  
\dot{C}(n)=C(n)B(n-1)-B(n)C(n).
\end{equation*}
Note that for $v(x)=x^2 +xt$ the relations \eqref{eq:res1} give
$$2\dot{B}(n) = [B(n),A]-I,\qquad  
2\dot{C}(n) = [C(n),A].$$
\end{example}
\begin{example} \normalfont
If  $\dot{v}(x) = x^2$, we obtain 
the non-Abelian Langmuir lattice
\begin{equation*}
\begin{aligned}
\dot{B}(n) &= B(n) C(n) - B(n+1) C(n+1) + C(n) B(n-1) - C(n+1) B(n),\\
\dot{C}(n) &=C(n) C(n-1) - C(n+1)C(n) + C(n) B(n-1)^2 - B(n)^2 C(n).
\end{aligned}
\end{equation*}
\end{example}


Next, we consider a weight $W(x,\vec{t})$ as in \eqref{eq:weight-deformation} with a multi-time Toda deformation, namely with a polynomial $v$ of the form
$$v(x,\vec{t})=v(x,t_1,\ldots,t_k)=\sum_{j=1}^k t_j x^j.$$
If we denote by  $\, \dot{}\, $  the derivative with respect to $t_j$, then we have $\dot{v}(L)=L^j$. Theorem \ref{eq:weight-deformation} gives the expressions for the derivatives of the recurrence coefficients, but if $j$ is large, then the coefficients $(\dot{v}(L))_{-1,-2}$ can be difficult to compute, and a much more convenient formulation is given as a Lax pair.

In the spirit of \cite[\S 2.8]{Ismail}, we identify the operator $L$ with the block tridiagonal matrix with block entries $(L_{nm})$, $L_{n,n+1}=I$, $L_{n,n}=B(n)$, $L_{n,n-1}=C(n)$, and $L_{n,m}=0I$ if $|n-m|\geq 2$. For a $N\times N$-block semi-infinite matrix $S=(S_{nm})$, 
we define $S_+$ as the matrix obtained by replacing all the $N\times N$ blocks of $S$ below the main diagonal by zero. Analogously, we let $S_-$ to be the matrix obtained by replacing all the $N\times N$ blocks above the subdiagonal by zero.

Then, we have the following result:
\begin{theorem} For $j=1,\ldots, k-1$, we have
\begin{equation*}\label{multiToda}
\dot{L}=\left[L,(L^j)_+\right]=-\left[L,(L^j)_-\right].
\end{equation*}
\end{theorem}
\begin{proof}
We first observe that
\begin{equation*}
\left[L,(L^j)_+\right]+\left[L,(L^j)_-\right]
=
\left[L,(L^j)_++(L^j)_-\right]
=
\left[L,L^j\right]=0,
\end{equation*}
which proves the second equality. Using that $\dot{B}(n)=(\dot{L})_{n,n}$ and $\dot{C}(n)=(\dot{L})_{n,n-1}$, we will complete the proof by showing that  $\left[L,(L^j)_+\right]_{n,n}$ equals the right hand side of the first equation in \eqref{eq:lattices}, that $\left[L,(L^j)_+\right]_{n,n-1}$ equals the right hand side of the second equation of \eqref{eq:lattices} and that $\left[L,(L^j)_+\right]_{n,m}=0$ otherwise. 

Note that $(\dot{v}(L))_m(n)=(L^j)_{n,n+m}$ for any indices $m,n$ so the first equation in \eqref{eq:lattices} reads
\begin{equation*}
\dot{B}(n)=(L^j)_{n,n-1}-(L^j)_{n+1,n},\qquad n\geq 1.
\end{equation*}

On the other hand, bearing in mind that $L$ is block tridiagonal and $L^j_-$ is lower triangular with zeros on the diagonal, we have
\begin{equation*}
\left[L,(L^j)_-\right]_{n,n}
=
L_{n,n+1}(L^j)_{n+1,n}-(L^j)_{n,n-1}L_{n-1,n}
=
(L^j)_{n+1,n}-(L^j)_{n,n-1},
\end{equation*}
since $L_{n,n+1}=I$ for any $n\geq 0$, which proves the result for the main diagonal. The second equation in \eqref{eq:lattices} is
\begin{equation*}
\dot{C}(n)=
(L^j)_{n,n-2}-(L^j)_{n+1,n-1}+(L^j)_{n,n-1}B({n-1})-B(n)(L^j)_{n,n-1},\qquad n\geq 2.
\end{equation*}

We also have
\begin{equation*}
\begin{aligned}
\left[L,(L^j)_-\right]_{n,n-1}
&=
L_{n,n}(L^j)_{n,n-1}+L_{n,n+1}(L^j)_{n+1,n-1}\\
&-
(L^j)_{n,n-1}L_{n-1,n-1}-(L^j)_{n,n-2}L_{n-2,n-1}\\
&
=B(n)(L^j)_{n,n-1}+(L^j)_{n+1,n-1}\\
&-
(L^j)_{n,n-1}B({n-1})-(L^j)_{n,n-2},
\end{aligned}
\end{equation*}
which proves the result for the first subdiagonal.

Finally, if $k\geq n+1$ we repeat the calculation using $\left[L,(L^j)_-\right]_{n,k}$, which gives $0$, consistently with $(\dot{L})_{n,k}$, and if  $k\leq n-2$, we compute $\left[L,(L^j)_+\right]_{n,k}$, which gives $0$ on both sides again.
\end{proof}

We remark that the multitime Toda lattice \eqref{multiToda} coincides with the one given in \cite[Proposition 4.4]{BFGA}.


\appendix
\section{Comparison of ladder operators}

In this article we have taken a different approach to ladder operators for matrix valued
orthogonal polynomials than for example the one in \cite{DI}. Their approach is inspired by
 \cite{ChenIsmail} and  \cite{WI} for the scalar orthogonal polynomials. This appendix
is meant to compare our approach with theirs for our class of weight functions, i.e. the
exponential weights in \eqref{eq:expweight}.

The ladder relations for exponential weights are stated in 
Section \ref{sec:exponential-weights} and can be formulated as
\begin{equation*}
P'(x,n) = \sum_{j=1}^{k-1}A_{-j}(n)P(x,n-j) + \left[A,P(x,n)\right],
\end{equation*}
for monic polynomials $P$. For exponential weights on the real line, we have the following identity:
\begin{equation}\label{eq:pear}
W'(x) = -W(x)V(x), \qquad
V(x)=v'(x)I-A^*-\rho(x),
\end{equation}
where $\rho(x)=W^{-1}(x)AW(x)$.
The ladder relation given in \cite{DI} for exponential weights reads
\begin{equation*}
P'(x,n) = F(x,n)P(x,n)-E(x,n)P(x,n-1),
\end{equation*}
where the coefficients are
\begin{align}\label{eq:EF}
E(x,n) \cH(n-1)&=
-\int_{\mathbb{R}} P(y,n) W(y)\frac{V(x)-V(y)}{x-y} P(y,n)^\ast  \, dy,\\ \nonumber
F(x,n) \cH(n-1)&=
\nonumber
-\int_{\mathbb{R}} P(y,n) W(y)  \frac{V(x)-V(y)}{x-y} P(y,n-1)^\ast \, dy.
\end{align}
These identities are obtained in the following way: we expand the derivative of $P(x,n)$ in the basis of MVOPs, with coefficients multiplying on the left:
\begin{align*}
P'(x,n) &= \sum_{k=0}^{n-1} \, \langle P'(x,n) , P(x,k)\rangle \cH(k)^{-1} P(x,k)\\
&= \sum_{k=0}^{n-1} \left(  \int_{\mathbb{R}} P'(y,n) W(y)P(y,k)^\ast \, dy\right) \cH(k)^{-1} P(x,k) \\
&= \int_{\mathbb{R}} P'(n,y) W(y) \left(\sum_{k=0}^{n-1} P(y,k)^\ast \cH(k)^{-1} P(x,k)\right) dy. \nonumber
\end{align*}
We integrate by parts, and the boundary terms vanish because of the decay of $W(x)$ at $\pm\infty$. This, together with \eqref{eq:pear}, gives
\begin{equation*}
\begin{aligned}
P'_n(x) &=  
- \int_\mathbb{R}  P(n,y)  W(y)  (-V(y)) 
\sum_{k=0}^{n-1} P(y,k)^\ast \cH(k)^{-1} P(x,k) dy\\
&=  
- \int_\mathbb{R}  P(n,y)  W(y) (V(x)-V(y))
\sum_{k=0}^{n-1} P(y,k)^\ast \cH(k)^{-1} P(x,k) dy,
\end{aligned}
\end{equation*}
where we have used the fact that the integral with $-V(x)$ vanishes by orthogo\-na\-lity. If we now apply the Christoffel-Darboux formula
\begin{equation}\label{eq:CD}
\begin{aligned}
(x-y)
\sum_{k=0}^{n-1} P(y,k)^\ast \cH(k)^{-1} P(x,k) 
&=  
P(y,n-1)^\ast \cH(n-1)^{-1}  P(x,n)\\
&- P(y,n)^\ast\cH(n-1)^{-1} P(x,n-1),
\end{aligned}
\end{equation}
we obtain the formulas \eqref{eq:EF} for the coefficients $E(x,n)$ and $F(x,n)$. Furthermore, using the formula for $V(x)$ in \eqref{eq:pear}, we can write 
\begin{align}\label{eq:EF2}
F(x,n) \cH(n-1)&=
-\int_{\mathbb{R}} P(y,n) W(y)S(x,y)
P(y,n-1)^\ast \, dy,
\\
\nonumber
E(x,n) \cH(n-1)&=
-\int_{\mathbb{R}} P(y,n) W(y)
S(x,y)
P(y,n)^\ast  \, dy,
\end{align}
where
\[
S(x,y)=\frac{v'(x)-v'(y)}{x-y}-\frac{\rho(x)-\rho(y)}{x-y} 
\]
On the other hand, by direct computation using the fact that $P(x,n)$ is monic and \eqref{eq:pear}, we have
\begin{equation*}
\begin{split}
&P(x,n)A -AP(x,n)\\
&= \sum_{k=0}^{n-1} \langle P_n A,P_k\rangle\cH(k)^{-1}P_k(x)\\
&=
\sum_{k=0}^{n-1} \left(\int_{\mathbb{R}} P(y,n)  W(y)\left(\rho(y)-\rho(x)\right) P(y,k)^\ast dy\right) \cH(k)^{-1}P(x,k).
\end{split}
\end{equation*}
Therefore, applying \eqref{eq:CD} again, we obtain
\[
\begin{split}
&P(x,n)A -AP(x,n)\\
&=
\left(
\int_{\mathbb{R}}
P(y,n)  W(y)\frac{\rho(y)-\rho(x)}{x-y}P(y,n-1)^*dy\right)\cH(n-1)^{-1}P(x,n)\\
&+
\left(
\int_{\mathbb{R}}
P(y,n)  W(y)\frac{\rho(y)-\rho(x)}{x-y}P(y,n)^*dy\right)
\cH(n-1)^{-1}P(x,n-1).\
\end{split}
\]
Comparing this last equation with \eqref{eq:EF2}, we find a relation between the two ladder operators, since
\begin{align*}
(F(x,n)+
E(x,n)) \cH(n-1)
&=
P(x,n)A -AP(x,n)\\
&-
\int_{\mathbb{R}} P(y,n) W(y)
\frac{v'(x)-v'(y)}{x-y} 
P(y,n-1)^\ast \, dy\\
&-
\int_{\mathbb{R}} P(y,n) W(y)
\frac{v'(x)-v'(y)}{x-y} 
P(y,n)^\ast \, dy.
\end{align*}

Lastly we note that $E(x,n)$ and $F(x,n)$ are explicit for the case treated in
Section \ref{sec:HermitePearson}. In principle one could derive the expressions 
using Proposition \ref{prop:PearsonHermite}, but it is simpler to deduce them
from Proposition \ref{prop:Pearson_hermite}:
\begin{align*}
F(x,n) &= -\cH(n)A^\ast \cH(n-1)^{-1}, \\
E(x,n) &= -x \cH(n)A^\ast \cH(n-1)^{-1}  
-n I
\\
&\qquad + \frac12 \cH(n)A^\ast \cH(n-1)^{-1} A
+ \frac12 \cH(n) ( A^\ast )^2 \cH(n-1)^{-1}.
\end{align*}

\bibliographystyle{plain}
\bibliography{biblio}

\end{document}